\documentclass{amsart}

\providecommand{\abs}[1]{\lvert#1\rvert}

\def \R{\mathbb{R}}

\def\eps{{\varepsilon}}

\def\supp{{\rm supp}}

\def\Diff{{\rm Diff}}

\def\Prob{{\mathbb{P}}}

\def\EXP{{\mathbb{E}}}

\newcommand{\vol}{\operatorname{vol}}

\def\reals{\mathbb{R}}

\def\brs{{\bar s}}

\def\cB{\mathcal{B}}

\def\cD{\mathcal{D}}

\def\cI{\mathcal{I}}

\def\cF{\mathcal{F}}

\def\cG{\mathcal{G}}

\def\cE{\mathcal{E}}

\def\cK{\mathcal{K}}

\def\cL{\mathcal{L}}

\def\cN{\mathcal{N}}

\def\cU{\mathcal{U}}

\def\cX{\mathcal{X}}

\def\cZ{\mathcal{Z}}

\def\fA{\mathfrak{A}}

\def\fU{\mathfrak{U}}

\def\hmu{{\hat\mu}}

\def\tG{{\tilde G}}

\def\tf{{\tilde f}}

\def\tg{{\tilde g}}

\makeatother

\def\beq{\begin{equation}}
\def\eeq{\end{equation}}

\usepackage{xcolor}

\usepackage[margin=3cm]{geometry}

\usepackage[normalem]{ulem}
\usepackage{amsthm}
\usepackage{amsmath, amssymb,mathtools}
\usepackage{mathrsfs}
\usepackage{hyperref}
\usepackage{graphicx, psfrag, epstopdf}
\usepackage{url}
\usepackage{enumitem}
\usepackage{comment}

\usepackage[backend=biber,sorting=nty,sortcites=true,style=alphabetic,
doi=true,natbib=true,isbn=false,url=false,maxbibnames=99]{biblatex}
\renewbibmacro{in:}{}

\DeclareBibliographyCategory{badbreaks}
\addtocategory{badbreaks}{goldsheid1989lyapunov}

\AtEveryBibitem{
  \ifcategory{badbreaks}
    {\defcounter{biburllcpenalty}{9000}}
    {}}

\theoremstyle{definition}
\newtheorem{definition}{Definition}[section]
\newtheorem{theorem}[definition]{Theorem}
\newtheorem{lemma}[definition]{Lemma}
\newtheorem{corollary}[definition]{Corollary}
\newtheorem{proposition}[definition]{Proposition}

\newtheorem{conjecture}[definition]{Conjecture}
\newtheorem{remark}[definition]{Remark}
\newtheorem{example}%[definition]
{Example}%[section]

\numberwithin{equation}{section}

\def\DS{\displaystyle}

\newcommand{\Aut}{\operatorname{Aut}}
\newcommand{\Id}{\operatorname{Id}}
\newcommand{\GL}{\operatorname{GL}}
\newcommand{\SL}{\operatorname{SL}}

\newcommand{\mc}[1]{\mathcal{#1}}
\newcommand{\mf}[1]{\mathfrak{#1}}

\newcommand{\E}[1]{\mathbb{E}\left[{#1}\right]}

\newcommand{\PP}{\mathbb{P}}
\newcommand{\N}{\mathbb{N}}
\newcommand{\Z}{\mathbb{Z}}
\newcommand{\Gr}{\operatorname{Gr}}
\newcommand{\SO}{\operatorname{SO}}

\newcommand{\wt}[1]{\widetilde{#1}}
\newcommand\dd{{\mathfrak{d}}}

\title[Conservative Coexpanding on Average Diffeomorphisms]{Conservative Coexpanding on Average Diffeomorphisms}
\author{Jonathan DeWitt and Dmitry Dolgopyat}
\address{Department of Mathematics, The University of Maryland, College Park, MD 20742, USA}
\email{dewitt@umd.edu, dolgop@umd.edu}
\date{\today}

\begin{document}
\begin{abstract}
 We show that the generator of a conservative  IID  random system whose dynamics expands on average codimension $1$ planes
has an essential spectral radius strictly
smaller than $1$ on Sobolev spaces of small positive index index. Consequently, such a system
has finitely many ergodic components. If there is only one component for each power of the random system,
then the system enjoys multiple exponential mixing and the central limit theorem. 
Moreover, these properties are stable under small perturbations.

As an application we show that many small perturbations of random homogeneous systems are exponentially mixing.
\end{abstract}

\maketitle

\section{Introduction}
\subsection{Overview of the main results.}
\label{SSMainResults}
In this paper we provide sufficient conditions for exponential mixing  of the IID random dynamics on higher dimensional smooth manifolds.

We now explain our main hypothesis. 
We say that a measure $\mu$ on $\Diff^1(M)$ is \emph{coexpanding on average} if there exists $N\in \mathbb{N}$ and $\lambda>0$ such that for all $x\in M$ and $\xi\in T^{1*}_xM$, the unit cotangent bundle, 
\begin{equation}\label{eqn:coexpanding_on_average}
\int N^{-1}\ln \| (D_xf^*)^{-1} \xi\|\,d\mu^N({f})\ge \lambda>0.
\end{equation}
 Here and below $Df^*$ denotes the adjoint (pullback) action on the cotangent bundle: if $\xi\in T_{fx}^*M$
and $v\in T_x M$, then 
$\displaystyle\langle D_x f^* \xi, v\rangle=\langle \xi, D_x f v\rangle $; and
$\mu^N$ denotes the law of $N$-fold composition of independent maps with law $\mu$. As we shall see below, coexpansion of $\mu$ directly implies good properties for the dynamics of $\mu^{-1}$, where $\mu^{-1}$ is the law of $f^{-1}$ when $f$ is distributed according to $\mu$. 
For this reason, we will also state results when $\mu^{-1}$ is coexpanding on average. Our main result is the following:

\begin{theorem}\label{thm:essential_spectral_gap}
Let $M$ be a closed Riemannian manifold and $\mu$ be a compactly supported probability measure on $\Diff^{\infty}_{\vol}(M)$ that is coexpanding on average. Then there exists $s_0>0$ such that for $s\in  (0,s_0]$ the associated generator
$\mc{G}\colon H^{s}(M)\to H^{s}(M)$
defined by 
$ (\cG\phi)(x)=\int \phi(f x) d\mu(f)$
has essential spectral radius less than $1$. If instead $\mu^{-1}$ is coexpanding on average, then the associated generator 
$\mc{G}\colon H^{-s}(M)\to H^{-s}(M)$
defined by 
$ (\cG\phi)(x)=\int \phi(f x) d\mu(f)$
has essential spectral radius less than $1$. 
\end{theorem}

Theorem \ref{thm:essential_spectral_gap} implies a variety of additional 
results. 
We say that that the random system is {\em totally ergodic} if for each natural number $q$ there is no non-trivial function
$\phi$ which is invariant for $\mu^q$ almost every $f$.
We show that for systems with $\mu^{-1}$ coexpanding on average 
the manifold $M$ decomposes into a finite number of totally ergodic components (Theorem \ref{cor:coexpanding_components}). If we assume that both $\mu$ and $\mu^{-1}$ are coexpanding on average, we are able to improve Theorem \ref{thm:essential_spectral_gap} to an essential spectral gap on $H^s$ for all $s\in [-s_0,s_0]$ (Theorem \ref{thm:essential_spectral_gapL2}). An important property of the coexpanding condition above is that if $\mu$ is coexpanding on average, then the $k$-point motion of $\mu$ is also coexpanding on average for all $k\in \mathbb{N}$ (Corollary ~\ref{CrEoAk}).

The above results pertain to
the essential spectral gap: they do not show exponential mixing yet. The problem is that the spectral argument does not give ergodicity. However, for many examples, ergodicity is already known. 
In that case, we obtain an actual spectral gap, which has a number of dynamical consequences.
In Section \ref{ScBack} we derive the following result.

\begin{corollary}
\label{CrSpGap-MEM}
If $\mu^{-1}$ is coexpanding on average totally ergodic measure on $\Diff^{\infty}_{\vol}(M)$, then the random dynamics of $\mu$ is multiply exponentially mixing and satisfies
the annealed central limit theorem. The same properties hold for small perturbations of~$\mu$.
\end{corollary}

 For the rest of the \S \ref{SSMainResults} we will say for brevity that a random system is {\em strongly chaotic}
if it enjoys the properties of the above corollary, that is, it is multiply exponentially mixing and satisfies
the annealed central limit theorem, and these properties are stable with respect to small
perturbations.

 It turns out (see Proposition \ref{prop:coa_d-1_planes}) that in the conservative setting 
 the coexpansion property is the same as the much more studied {\em expansion of average on codimension
one planes} property. Therefore we can use previous work on this subject,
including {\em the invariance principle} of Avila-Viana, to verify the coexpansion on average condition
in many examples. In particular, we obtain the following statement.

\begin{theorem}
\label{ThExpMixGen}
 Suppose that $\dim M>1$ and let $\cU\subseteq \Diff^\infty_{\vol}(M)$ be an open set consisting of uniformly $C^r$ bounded volume preserving 
diffeomorphisms where $r=r(\dim(M))$ is a sufficiently large
constant. Then the set of measures on $\cU$ so that the corresponding random dynamics is 
 strongly chaotic
 contains weak* open and dense subset. 
\end{theorem}

The proof of Theorem \ref{ThExpMixGen} relies on the following ingredients:
\begin{enumerate}
\item 
The set of coexpanding on average measures is open and dense.
\item 
Exponential mixing is dense.
\item 
Exponential mixing is open among coexpanding on average systems.
\end{enumerate}

The first two ingredients are due to \cite{elliott2023uniformly} (see also \cite{BCZG} for some related results).
Our contribution is the third ingredient, which relies on
essential spectral gap given by Theorem \ref{thm:essential_spectral_gap} and stability of the peripheral spectrum, given by Keller-Liverani
stability theory in \cite{KL}.

\begin{remark}
\label{RmNotOpen}
We note that neither ergodicity nor exponential mixing are open by themselves. Indeed,
take $M=\mathbb{T}^d$ and let $\mu$ be a random translation $x\mapsto x+\alpha$ where $\alpha$ is uniformly
distributed on $\mathbb{T}^d.$ Then $\mu$ is exponentially mixing (in fact, the points of the orbit are IID uniformly
distributed on $\mathbb{T}^d$). 
Let $\mu_Q$ be defined similarly but now $\alpha$ be uniformly distributed on rational vectors with denominator 
$Q.$ Then $\mu_Q$ is not ergodic. Thus the openness
comes by combining ergodicity with coexpansion
on average.
\end{remark}

Theorem \ref{ThExpMixGen} allows us to produce coexpanding on average random systems but the size of the support of their generator can be arbitrary large.
It is of great interest to study coexpansion on average for tuples of fixed size, where the random dynamics is generated by the uniform measure on the elements of the tuple. 
As was mentioned above, we can use classical techniques for producing hyperbolicity for random systems
to provide such examples. 
In particular, we shall show that many  homogeneous systems, as well as their perturbations, satisfy this condition. 
This is discussed in detail in Section \ref{sec:examples}. Here, we provide several representative examples.
Note that the words {\em coexpanding on average} do not appear explicitly in the statements below.

\begin{corollary}\label{cor:main_cor}
(a) Let $(A_1,\ldots, A_m)$ be a tuple of $\SL_d(\mathbb{Z})$ matrices 
generating a Zariski dense subgroup of $\SL_d(\mathbb{R}).$
Let $(f_1, \dots ,f_m)$ be either of the following systems:

\begin{enumerate}
\item[(i)] 
$M=\mathbb{T}^d=\R^d/\Z^d$ and $f_j(x)=A_j x+b_j$ for some vectors $b_j$;

\item[(ii)]
 $M=\SL_d(\reals)/\Gamma$ where $\Gamma$ is a uniform lattice in $\SL_d(\reals)$ and
$f_j  x$=$A_j x.$
\end{enumerate}

Let $\tf_j$ be small, smooth, volume preserving perturbations of $f_j.$ Then the random system generated by $(\tf_1, \dots ,\tf_m)$ is strongly chaotic
and its generator
 has spectral gap on $L^2$.
\vskip1mm

(b) Suppose that $d$ is even and 
let $(R_1, \ldots, R_m)$ be rotations of the sphere ${S}^d$ generating a dense subgroup of $\SO_{d+1}(\mathbb{R}).$
Let $\tf_j$ be small,  smooth, conservative perturbations of $f_j.$ Then either the random system generated by $(\tf_1, \dots \tf_m)$  is strongly chaotic
and its generator
 has spectral gap on $L^2$
or the $\tf_j$ are simultaneously conjugated to rotations.
\end{corollary}

 It is possible that the same results hold in odd dimensions as well but we are unable to prove this
(the reason for this is discussed in  \S \ref{SSErgEoA}). 
 However, we have the following partial result. Recall that an {\em isotropic manifold} is a rank $1$ symmetric space of compact type of dimension at least $2$. The full list of such manifolds includes $S^d$, $\mathbb{RP}^d$, $\mathbb{CP}^d$, $\mathbb{HP}^d$, and the Cayley projective plane.

\begin{theorem}
\label{ThGenericNearIsom}
Let $M$ be an isotropic manifold and
$m\ge 3$, then there exists an open neighborhood $\mc{U}$ of $\mathrm{Isom}(M)^m$ 
in the space of $m$-tuples in
$\Diff_{\vol}^\infty(M)^m$ 
such that on an open and dense set in $\mc{U}$ the corresponding dynamics is strongly chaotic and enjoys a spectral gap in $L^2$.
\end{theorem} 

In contrast, the question of when the random isometries of $M$ enjoy spectral gap in $L^2$ is wide open. 
The first examples were constructed in \cite{Mar80, Sul81}. \cite{BourgainGamburd} proves density of the
spectral gap for random symmetric $SU(2)$ actions.
\cite{Fisher-IMRN} shows that the spectral gap has probability zero or one in $SU(2)^m$, but it is unknown 
which alternative holds. Theorem \ref{ThGenericNearIsom} shows that the question is much easier for
small perturbation of isometries.
\subsection{Related Results}

\subsubsection{Expanding on average}
 The expanding on average condition first appeared in the study of IID matrix products and shows up naturally for the following reason. Suppose $\mu$ is a probability measure on $\SL(d,\R)$, and we study the Lyapunov exponents of the associated random walk. 
 Consider a $\mu$-stationary measure $\nu$ for the induced 
 random walk on $\mathbb{RP}^{d-1}$. For a matrix $A$ and a unit vector $v$, define $\Phi(A,v)=\ln \|Av\|$. Then consider the integral:
\[
\iint \Phi(A,v)\,d\nu(v)\,d\mu(A).
\]
According to Furstenberg's formula, the values that this integral takes for different stationary measures $\nu$ are a subset of the Lyapunov exponents. Further, if there is a unique stationary measure $\nu$, then the integral is always equal to $\lambda_1(\mu)$, the top Lyapunov exponent.  Moreover, if the stationary measure is unique, we have uniform convergence of Birkhoff sums against $\Phi$. Namely,
\[
\lim_{n\to \infty} \E{\ln\|A^n_{\omega} v\|}\to \lambda_1(\mu)
\]
uniformly independent of $v$. Hence we obtain the expanding on average condition as long as $\lambda_1(\mu)$ is positive. For a more detailed discussion see \cite[Cor.~III.3.4]{BougerolLacroix} and \cite[Ch.~6]{viana2014lectures}.

The Lyapunov exponent results of Furstenberg were extended to random dynamical systems and beyond in
\cite{AvilaViana10, BarMal20, Bax86, Bax89,  brinkifer, Car85, Crauel90,  kifer1986ergodic, Led84, Led86} and others.

The expanding on average condition was applied to studying ergodic properties of random dynamical systems
 in \cite{BS88, DKK04}. An application to stable ergodicity appears in
 \cite{dolgopyat2007simultaneous}, which proved  stable ergodicity of certain random isometric systems. 
 This property is also crucial for stable ergodicity results of the present paper.

 The application of expansion on average to the mixing of random systems appears in \cite{DKK04} and was expanded in \cite{BCZG, BFPS}. The latter paper obtains mixing results
 similar to ours under stronger conditions. (The expansion on average is not explicitly assumed in \cite{BFPS}
 but they refer to other papers such as \cite{BCZG} for the verifications of the assumptions of their main theorem
 in specific models, and the first (among many) steps of such verification 
 usually amounts to expansion on average.) Roughly speaking \cite{BFPS} use similar ideas to handle high
 frequencies, but they use PDE techniques to treat low frequencies, while we rely on ergodic theoretic approach
 which seems more flexible. Thus we can obtain similar conclusions under less restrictive assumptions. 
 On the other hand \cite{BFPS} do not assume the independence of the consecutive maps, they work with more
 general Markov chains. Similar extensions seem possible in our setting as well, but it would make the arguments less transparent.

Later, interest in expansion on average
increased when it was realized that the condition should be generic and also leads to a variety of interesting results. Perhaps most surprising were measure rigidity results obtained by Brown and Rodriguez Hertz in \cite{BrownRodriguezHertz}, which showed, in particular, that for a volume preserving expanding on average random dynamical system on a surface all stationary measures are invariant, and all invariant measures are either periodic or volume. 
Cantat and Dujardin studied random walks on complex surfaces in \cite{cantat2025hyperbolicity} and gave concrete criteria for this random walk to be expanding on average. In particular, they then apply this result to classify the stationary measures for these surfaces. See also \cite{cantat2024dynamics} for additional perspective on this application.
Quite recently, the results of \cite{BrownRodriguezHertz} were generalized to higher dimensions under the condition of being expanding on average in all dimensions \cite{brown2025measure}, plus additional assumptions  such as all Lyapunov exponents being non-zero.\footnote{We note that the conditions of \cite{brown2025measure} also imply ergodicity so that paper provides  additional examples of expanding on average systems that are ergodic.
The ergodicity plays important role in the applications of our results described in Section \ref{ScApplications}.}
Also in \cite{liu2016lyapunov} large deviations were studied for expanding on average systems. Chung constructed some discrete perturbations of the standard map 
and gave some alternate characterizations of the expanding on average condition \cite{chung2020stationary}. This generalized perturbations due to Blumenthal, Xue, and Young that used continuous noise \cite[Prop.~9]{blumenthal2018lyapunov}, \cite{blumenthal2017lyapunov}. A generalization 
 of the expanding on average condition
was also used by Eskin and Lindenstrauss in the homogeneous setting \cite{eskin2018random}. As will be discussed more below, Potrie \cite{potrie2022remark} showed how one could construct more examples of expanding on average systems on surfaces and that these systems are dense in weak* sense. Later, in \cite{dewitt2024expanding}, the authors showed that conservative expanding on average random dynamics on surfaces satisfies quenched exponential mixing. In the dissipative setting, an important question is the existence of an absolutely continuous invariant measure. In \cite{brown2024absolute}, Brown, Lee, Obata, and Ruan showed that for dissipative perturbations of a pair expanding on average pair of Anosov diffeomorphisms there exists an absolutely continuous stationary measure. 

The above mentioned work is, in the non-homogeneous case, limited to surfaces. In higher dimensions much less is known.  An important work 
by Elliott Smith \cite{elliott2023uniformly}  implies that the expanding on average condition, and its generalization to $k$-planes, is weak* dense in the space of driving measures. 

 \subsubsection{Contracting on average diffeomorphisms} 

 The importance of expanding on average condition for IID  matrix products is that it is equivalent to
the fact that the induced projective action is contracting on average. This fact was crucial in the study 
of statistical properties of random matrix products, see \cite{GR85, LePage82, LePage89}, 
and led to a general theory of contracting on average systems, see
\cite{Antonov,  barrientos2024contracting,  Blank01, Kaijser, Malicet17, Stenflo}.
We note that similarly to the present work, quasicompactness of the associated transfer operators
plays a key role in most of the above mentioned papers. However, since contractions improve regularity,
in the mostly contracting case one can get quasicompactness on the spaces of smooth (H\"older) functions,
while in the present case one needs to work with less regular functions which introduces additional complications.
We emphasize that unlike the contracting on average property, whose random dynamics are extremely similar to that of an actual contraction, 
conservative expanding on average maps look much more like maps that have at 
least one positive and one negative Lyapunov exponent. In this sense the dynamics looks partially hyperbolic. 

\subsubsection{Generic dynamics.}
For deterministic systems, establishing even weak statistical properties is quite difficult whereas for random systems this is much easier. If the random dynamics is sufficiently rich, then many statistical properties can be shown. This was done in  \cite{DKK04} in the context of stochastic flows. 
A recent work of Blumenthal, Coti Zelati, and Gvalani shows exponential mixing of some random flows including the Pierrehumbert model \cite{BCZG}. An important question is just how ``rich" the random dynamics must be in order to exhibit chaotic behavior. The following conjecture appears in \cite{dolgopyat2007simultaneous}.\footnote{ This statement is the strengthened statement that the result hold for pairs---and not longer tuples---which was demanded by the audience during the first author's talk at the 2024 Penn State Fall Conference.}

\begin{conjecture}
For each closed manifold $M$ and regularity class $k\ge 1$, the expanding on average pairs $(f,g)$ are open and dense in $\Diff^k_{\vol}(M)\times \Diff^k_{\vol}(M)$. 
\end{conjecture}

 \noindent Naturally the idea of the conjecture is that it should take very little randomness for a random system to have strong properties.

  Similarly, \cite{dewitt2024expanding} conjecture that for a generic tuple the associated random dynamics
 is exponentially mixing. The present work shows that the two conjectures are intimately related. See 
 Proposition \ref{prop:stable_exponential_mixing} for the precise statement.
 
 The foregoing discussion was mostly limited to understanding dynamics in the  random conservative setting. However, there are some notable results in the dissipative setting as well, see
 \cite{barrientos2024contracting, LeJan86, DKK04}.

 We also note that the Smale and Palis conjectures \cite{Smale, Palis} 
 about the genericity of good behaviors from either the  topological 
 or ergodic theoretic point of view, were motivated by the success in understanding 
 hyperbolic systems. In fact, a reasonable description could be obtained for large classes of 
 nonuniformly hyperbolic systems \cite{BCS, dolgopyat2016geometric}. The problem with deterministic systems
 is that they could admit small invariant regions where dynamics is 
 far from hyperbolic \cite{Berger16, Berger17, Newhouse79}. 
 The fact that hyperbolicity is much more prevalent in the random setting because non-hyperbolicity implies existence
 of invariant geometric structures motivates a quest to understand the behavior of generic 
 random systems.

\subsection{Comments on the proof}
 There are two main approaches to establishing exponential mixing for systems without a large symmetry group.
The first, more classical, approach is based on quasicompactness in an appropriate space. It goes back to the
work of Lasota--Yorke \cite{LY73} and Ruelle \cite{RuelleTF} and
requires establishing Lasota--Yorke type inequalties (see \cite{BaladiBook, ParryPollicott,  viana1999lectures}). 
This approach got a powerful boost in the last two decades with the development of
 weighted Banach spaces \cite{AGT06, BaladiBook, BKL, 
CastorriniLiverani, GouezelLiverani06, Tsujii01}, which led to powerful results in the deterministic setting. 
 In the present paper we also follow this approach. 
 The argument is most similar to the arguments of \cite{Tsujii23} and \cite{BFPS}, which also work directly with the symbol.
While deterministic systems may require the use of an anisotropic Banach space that is well adapted to the dynamics of the system, in our case because of the uniformity of the assumption we are able to work directly with the simplest Hilbertian spaces\footnote
{The action on Sobolev spaces was also considered in deterministic setting, see \cite{Thomine, Tsujii23}. However, in those cases it was used that if the deterministic system is expanding then the inverse system is contracting, so that one has essential spectral gap on Sobolev spaces with positive index. 
This is not the case for random systems, where there are many examples where both $\mu$ and $\mu^{-1}$ are expanding on average and, in fact, this is conjectured to happen generically.}
---the Sobolev spaces $H^s(M)$. 

The proof of our main results in this paper is quite different than in our earlier work \cite{dewitt2024expanding}, which 
used a coupling  method developed in \cite{Young99, dolgopyat00}.
The proof of \cite{dewitt2024expanding}
relied on a delicate argument to construct a coupling between two curves lying in our surface. The proof makes detailed use of Pesin theory and many tools from smooth dynamics. The consequences obtained are stronger as well: that paper is able to show that a $C^{1+\text{H\"older}}$-curve exponentially equidistributes. The methods in this paper do not yield such a result because a measure along a curve is not regular enough to be in $H^s$ for $s$ close to $0$.

The current 
proof proceeds by a direct calculation of the essential spectral radius that gives a relatively explicit relationship between the expansion on average constant and the spectral radius. In this sense, the argument is not a particularly dynamical one as it does not shed much light on how the dynamics comes to be mixing, whereas the argument in \cite{dewitt2024expanding} shows this quite explicitly.  
 On other hand, the analytic approach of the present paper makes it much easier to see how the system
changes under small perturbations\footnote{There is also an approach to perturbation theory 
based on coupling and shadowing, see \cite{ChernovBrownian, ChernovGalton, DolgopyatInv, DolgopyatMoscow}. 
However, the results obtained by this method are weaker than the results relying on analytic techniques.},
both when we change the diffeomorphisms, which entails spectral stability results elucidated in 
\S \ref{SSStability},
and when we apply a multiplication by a small function which allows one to obtain the Berry--Esseen bound 
of Theorem~\ref{thm:central_limit_theorem}(b).

On the other hand, the approach of \cite{dewitt2024simultaneous} seems less sensitive to the independence
assumption and so it may be easier to extend to the setting of partially hyperbolic skew products. 
We note that for partially hyperbolic systems  there are many results in the setting where all central 
exponents have the same sign. 
They were first studied in \cite{bonatti2000SRB,alves2000srb}, and \cite{dolgopyat00}. Later more properties were shown in \cite{andersson2010robust, viana2013physical,dolgopyat2016geometric}. 
The systems with mixed exponents in the center are much less understood, even though
their abundance was demonstrated in \cite{AvilaViana10},
and we hope that studying expanding and coexpanding on average systems could shed some light on their properties.

We also note that both coupling and analytic techniques only show mixing on small scales
and so they require mixing to start the argument. Namely,
in the case of coupling we need the two pieces of the curves to be close to start the coupling procedure,
while in the analytic case we only have good control of high frequencies, so we only get quasi-compactness
as opposed to the spectral gap. In the two dimensional case the mixing was already known due to
\cite{chung2020stationary, dolgopyat2007simultaneous}, but in the present setting it is not known in the full
generality, see \S \ref{SSErgEoA} for a detailed discussion.

The reader may notice that the present proof is significantly shorter than the proof in \cite{dewitt2024simultaneous}. Moreover, a significant part of the present paper is devoted to examples, with 
the proof of the main result being limited to Sections \ref{ScSymb-Op}--\ref{ScMainProof}.
The reason for this disparity is that  \cite{dewitt2024simultaneous} required novel finite time estimates
in Pesin theory which are of independent interest. In the present paper we can use the well developed theory
of pseudodifferential operators. This is the main reason why we assume that the our random maps 
are $C^\infty$. While this assumption is clearly not optimal it allows us to cite many references that
do not explicitly track the smoothness required for various estimates.\\

\noindent\textbf{Acknowledgments.} The first author was supported by the National Science Foundation under Award No.~DMS-2202967. The second author was supported by the National Science Foundation under award No.~DMS-2246983. The authors are grateful to 
Carlangelo Liverani for helpful discussions. After we had proved the main results of this paper, we learned from Zhiyuan Zhang that he had independently obtained a related proof using curvelet spaces, and we remain grateful to Zhiyuan for the ensuing discussions. 
In particular, the results of \S \ref{SSDissipative} were suggested by Zhiyuan. The authors are also grateful to Thibault Lefeuvre for pointing out that the direction of the dynamics was reversed in an earlier version of this article.

\section{Background}

Here we describe  the necessary background.

\subsection{Symbols and the Operators} 
First we describe, the symbol class $S^m(X)$ where $X\subseteq \R^n$ is an open set.

For a domain $X\subseteq\R^n$ a symbol of class $S^m(X)$, $m\in \R$ is a smooth function $a(x,\xi)\colon X\times \R^n\to \R$ so that on every compact set $K\subset X$ there exists $C_{\alpha,\beta}$ such that 
\begin{equation}\label{eqn:S_m_definition}
\abs{D^{\alpha}_{\xi}D^{\beta}_xa(x,\xi)}\le C_{\alpha,\beta}(1+\abs{\xi})^{m-\abs{\alpha}}.
\end{equation}
The corresponding symbol class $S^m(M)$ on a manifold is defined analogously
by means of charts, see \cite[Ch.~I.5]{treves1980introduction}.
In the language of Shubin, this is the class $S^m_{1,0}(X)$ \cite[Def.~I.1.1]{shubin2001pseudodifferential}. 
We write $\Psi^m(X)$ for the class of pseudodifferential operators on $X$ defined using symbols by the standard quantization in $\R^n$. 
The operators in $\Psi^{-\infty}(X)$ are called smoothing because they map $H^s\to C^{\infty}$ for all $s\in \R$.  

Write $\Psi^m(M)$ for the class of pseudodifferential operators whose restriction to any charts---up to perturbation by a smoothing operator in $\Psi^{-\infty}$---is pseudodifferential operator on the chart as described above. 
In particular
smoothing operators have symbol 0.
The principal symbol\footnote{ We do not provide a precise definition of the principal symbol here since it
 is not important for our purpose, we only use Lemma \ref{lem:square_root} below.}
of an operator is an element of $S^m(T^*M)$ and is well defined modulo $S^{m-1}(T^*M)$. If two pseudodifferential operators in $\Psi^{m}(M)$ have the same principal symbol then their difference is an operator in $\Psi^{m-1}(M)$. We write $\sigma_A(x,\xi)\colon T^*M\to \R$ for the principal symbol of a pseudodifferential operator $A$.
The association between symbols and pseudodifferential operators is given by a quantization procedure $\text{Op}$ that takes a 
a function on $T^*M$
 and produces a pseudodifferential operator  in $\Psi^m(M)$ with that principal
 symbol. For our purposes, we only need to  know that such a quantization procedure exists.

\subsection{The pushforward}
\label{SSPullback}
A useful construction is the pushforward of a pseudodifferential operator. If $A$ is a pseudodifferential operator in symbol class $S^m(M)$, and $f\colon M\to M$ is a smooth diffeomorphism, then the {\em pushforward} $A^f$ of $A$ 
acts on a function $\phi$ by 
\begin{equation}\label{eqn:pushforward_symbol}
A^f: \phi\mapsto (A(\phi\circ f))\circ f^{-1}.
\end{equation}
See the discussion surrounding \cite[Thm.~I.3.3]{treves1980introduction}, \cite[Sec.~I.4.2]{shubin2001pseudodifferential}, or \cite[Lem.~5.2.7]{lefeuvre2024microlocal}.
An important fact for the symbolic calculus of pseudodifferential operators is that the principal symbol is functorial with respect to the pullback and pushforward by diffeomorphisms. Namely, if $A$ has symbol $a(x,\xi)$, then $A^f$ has principal symbol 
$ a(f^{-1}(x), (D_x(f^{-1})^*)^{-1}(\xi))$ . This is why the symbol class $S^m(M)$ is well defined.

The random dynamics acts on pseudodifferential operators via the  pushforward. 
For an operator $\Psi$, we let $\mc{L}\Psi$ denote the averaged pushforward
\begin{equation}\label{eqn:averaged_operator_action}
(\mc{L}\Psi)(\phi)=\int (\Psi^{f_{\omega}})\phi\,d\mu(\omega),
\end{equation}
where we defined the pushforward of a pseudodifferential operator as above.
Note that this will preserve the symbol class of $\Psi$ if the $f_{\omega}$ lie in a compact subset of $\Diff^{\infty}(M)$. 

\subsection{Elliptic Operators}
We will have particular use for elliptic operators. For an open set $X\subseteq \R^n$,
we say that a symbol $\sigma(x,\xi)\in C^{\infty}(X\times \R^n)$ is {\em elliptic}  if for every compact subset $K\subseteq X$, there
are positive constants $C_1, C_2$ such that for all sufficiently large $\xi.$
\begin{equation}
  \label{EqElliptic}  
 C_1\abs{\xi}^m\le \abs{\sigma(x,\xi)}\le C_2\abs{\xi}^m. 
\end{equation}
As the principal symbol transforms appropriately under pushforward, this definition extends naturally to manifolds
 and the estimate \eqref{EqElliptic} holds there.
See \cite[Sec.~I.5]{shubin2001pseudodifferential} for more information.

\subsection{Sobolev Norms}
One can use pseudodifferential operators for defining the Sobolev spaces. In fact, there are several equivalent approaches to this. See a discussion in \cite[Prop.~I.7.3]{shubin2001pseudodifferential}.

Here we will just use the fact that for every $s\in \R$ that there exists pseudodifferential operator $\Delta^{s}$ with principal symbol $\|\xi\|^{s}$ such that $\Delta^{s}\colon H^{s}(
M)\to L^2(M)$ is an isometry, 
see e.g.~\cite[Lem.~II.2.4]{treves1980introduction}. Note that the notation for the pushforward of $\Delta^s$ by a diffeomorphism $f$ looks crowded: $(\Delta^s)^f$.

We can choose a particularly simple definition of the Sobolev norms. For $s>0$, one defines the $H^s$ Sobolev norms by 
\[
\|\phi\|_{s}^2=\|(\Id+\Delta^{s})\phi\|_{0}^2.
\]
For $s<0$, one can define them as 
\[
\|\phi\|_{s}^2=\|\Delta^{-s}\phi\|_0^2.
\]
Note that these definitions are basically the same, up to the $\Id$ term which is compact as a map $H^s\to L^2$. Also, compare with \cite[\S 5.3.2.1, 5.3.2.2]{lefeuvre2024microlocal}.
Shubin and Lefeuvre's definitions of the Sobolev  norms for $0<s<1$ are different but only by a 
compact error,  which is the quadratic form defined by a compact operator.

\subsection{Interpolation inequalities}\label{subsec:interpolation}

 We now review a useful fact concerning the interpolation of the spectral radius for an operator on an interpolation space.  For an operator $A\colon V\to V$ on a Banach space, we write $r_e(A)$ for its essential spectral radius. 

There are two main types of interpolation: real and complex interpolation. 
 Complex interpolation will be more useful for us.
In this case, one starts with a 
complex Banach couple, which is a pair $(A_0,A_1)$ of Banach spaces along with an embedding in a 
complex Hausdorff vector space. 
For each $\theta\in (0,1)$ one obtains an interpolation space, which we denote by $[A_0,A_1]_{[\theta]}$. 
For an overview of the general theory see \cite{bergh1976interpolation}.

The following result allows us to 
interpolate the norm, the spectral radius, and the essential spectral radius.

\begin{lemma}\label{lem:interpolation}
(\cite[Thm.~4.1.2]{bergh1976interpolation}, 
\cite[Prop.~5.2]{szwedek2015on}).
Suppose $(A_0,A_1)$ is a complex Banach couple, then
\begin{equation}
\|T\|_{[\theta]}\le \|T\|_{A_0}^{1-\theta}\|T\|_{A_1}^{\theta},
\end{equation}
and 
\begin{equation*}
    r_e(T\colon (A_0,A_1)_{[\theta]}\to (A_0,A_1)_{[\theta]})\le r_e(T\colon A_0\to A_0)^{1-\theta}r_e(T\colon A_1\to A_1)^{\theta}. 
\end{equation*}
Note that the estimates on the norm of the interpolation imply that 
we can interpolate the spectral radius because the spectral radius of an operator $A$ is equal to 
$\displaystyle \lim_{n\to\infty} n^{-1}\log \|A^n\|$. 
\end{lemma}

The use of this is that one can interpolate between Sobolev spaces
(\cite[Thm.~6.4.5]{bergh1976interpolation}, \cite[p.~22]{hamilton1975harmonic}). 
The pair of Sobolev spaces $(H^{s_0}, H^{s_1})$, $s_0,s_1\in \R$, form an interpolation couple and complex interpolation gives  
\begin{equation}\label{eqn:complex_interpolation}
[H^{s_0},H^{s_1}]_{[\theta]}=H^{\theta s_0+(1-\theta) s_1}.
\end{equation}

\begin{remark}
Similar results hold for real interpolation. In that case interpolation spaces depend on two 
parameters
$\theta\in (0,1)$ and $q\ge 1$, and
\begin{equation}\label{eqn:besov}
[H^{s_0},H^{s_1}]_{\theta,q}=B^{s^*}_{2,q},
\end{equation}
 where $s^*=(1-\theta)s_0+\theta s_1$ and 
$B_{pq}^s$ is the  Besov space $B^s_{pq}.$

Using this fact one can also obtain spectral gap on appropriate Besov 
spaces (see Remark \ref{rem:besov}), but this will be less useful for us, so we do pursue this subject in detail.
\end{remark}

\subsection{Weak mixing of random systems}
\label{SSWMSkew}
Random dynamics on a manifold $M$ is naturally encoded by a skew product 
$F$ on $\Sigma\times M$ where $\Sigma=\supp(\mu)^\mathbb{N}$.
It is
defined by 
\begin{equation}
\label{DefSkew}
    F(\omega, x)=(S\omega, f_{\omega_0}(x))
\end{equation}
where $S$ is the shift.
If $\nu$ is a stationary measure for the random dynamics given by a measure $\mu$ on $\Diff_{\nu}(M)$, then we say that a random system is \emph{weak mixing} on $L^2(\nu)$ if there does not exist a non-trivial function $\phi\in L^2(M,\nu)$ such that 
\begin{equation}
   \label{RandWM} 
\mathbb{E}_{\mu}\left[\phi\circ f\right]=e^{i\theta}\phi \text{ for }\theta\in \R.
\end{equation}
Note that this is implied by the usual skew product on $\Sigma\times M$ being weak mixing for the invariant measure $\mu\times \nu$.
 Indeed, without loss of generality we
 may assume that $\|\phi\|_{L^2}=1.$ Then taking the scalar product of both sides of \eqref{RandWM} with $\phi$ we obtain
$\int \langle \phi, \phi\circ f \rangle d\mu(f)=e^{i\theta}$
which is only possible if 
$$\phi\circ f=e^{i\theta} \phi \quad \text{for $\mu$ almost every }f.$$
\noindent The last equality shows that $\mu$ is weak mixing iff $\mu^{-1}$ is weak mixing.

\subsection{Perturbation of the essential spectrum.}
We recall a result of Keller and Liverani \cite{KL} that is convenient for studying the essential spectrum of perturbations. 
Let $(\cB, \| \cdot\|)$ be a Banach space.  Suppose that there is a second norm $|\cdot|$ on $\cB$ 
and a family of operators $\cG_\eps\colon \mc{B}\to \mc{B}$  indexed by $\varepsilon\ge 0$ and constants $\eta\in (0,1)$, $C, M>0$,
and a monotone upper semicontinuous function $\tau(\epsilon)$ satisfying the following conditions: 
\begin{equation}
 \label{UniGrowth}
\text{There exist } C, M\text{ such that for all } \eps,\,  \|\cG_\eps^n\|\leq C M^n;
\end{equation}
\begin{equation} \label{KL-LY}
 \|\cG_\eps^n \phi\|\leq C [\eta^n \|\phi\|+M^n |\phi|];
\end{equation}
\begin{equation}
 \label{FinMult}
\text{Spec}(\cG_\eps)\cap \{|\lambda|>\eta\} \text{ consists of isolated eigenvalues of finite
multiplicity}; 
\end{equation}
\begin{equation}
\label{TameCont} \text{For all } \phi\in\cB,\, 
 |\cG_\eps\phi-\cG_0\phi|\leq \tau(\eps) \|\phi\| 
\text{ where } \tau(\eps)\to 0 \text{ as }\eps\to 0.
\end{equation}

Fix $r>\eta$ and let  $V_{r, \delta}=\{\lambda: |\lambda|\geq r\text{ and } d(\lambda, \mathrm{Spec}(\cG_0))>\delta\}.$
The next result is a special case of \cite[Theorem 1 and Corollary 1]{KL}.

\begin{proposition}
\label{PrSpecCont}
Suppose \eqref{UniGrowth}--\eqref{TameCont}. Then  there exists $\theta,D>0$ such that for each $r, \delta$ there exists $\eps_1\leq \eps_0$, 
depending only on the constants fixed above, such that for $|\eps|\leq \eps_1$:
\begin{enumerate}[leftmargin=*]
\item[(i)]
$\cG_\eps$ has no eigenvalues in $V_{r,\delta}$;
\item[(ii)] 
The multiplicity of eigenvalues in each component of $(\mathbb{C}\setminus V_{r,\delta})\cap \{|\lambda|\geq r\}$ is constant;
\item[(iii)]
Each simple eigenvalue $\lambda_0$ of $\cG_0$ can be continued so that $|\lambda_\eps-\lambda_0|\leq D\tau(\eps)^\theta$.
\end{enumerate}
\end{proposition}

\subsection{Ergodicity}
\label{SSErgEoA}
Recall that for a random dynamical system, a stationary measure $\nu$ is \emph{ergodic} if it does not have any a.s.~invariant sets of intermediate measure. As was mentioned above, our results show essential spectral gap but do not show ergodicity. 
Ergodicity is not known to follow from just the coexpanding on average assumption or even expanding on average on all $k$-planes  defined in \S \ref{SSBundleDef} below. This is  due to a possible presence of zero Lyapunov exponents.

That said, it is possible to prove ergodicity with additional hyperbolicity assumptions. In \cite{dolgopyat2007simultaneous}, it was shown that knowing the expanding on average condition for all $k$-planes, combined with a lack of zero Lyapunov exponents is enough to deduce ergodicity for a random dynamical system. In particular, for conservative dynamics on a surface expanding on average dynamics is ergodic \cite[Sec.~6]{dewitt2024expanding}. The proof is given by a type of random Hopf argument where the role of the stable and unstable manifolds in the usual Hopf argument is replaced by the use of the stable manifolds for different realizations of the random dynamical system. See \cite{chung2020stationary} where this argument is explained in detail. A consequence of this approach to ergodicity is that the examples in \cite{dolgopyat2007simultaneous} are only known to be ergodic for even dimensional spheres, whereas the dynamics on odd dimensional spheres might have a zero Lyapunov exponent. This can happen due to the formula for the Taylor expansion of the Lyapunov exponents in  \cite[Thm.~2]{dolgopyat2007simultaneous}. This is why Corollary~\ref{cor:main_cor} requires even dimensional spheres.

We shall also use the following criterion for ergodicity of the random system, which follows from
\cite[Theorem 3]{Kakutani}  or \cite[Prop.~I.1.3]{Liu1995smooth}.

\begin{proposition} \label{PrRandomET}
The following properties are equivalent:

\begin{enumerate}[leftmargin=*]
    \item[(a)] The skew product defined by \eqref{DefSkew} is not ergodic.
    \item[(b)]  There exists a measurable set $\Omega\subset M$ with $0<\nu(\Omega)<1$
which is  invariant $\text{mod } 0$  for $\mu$ almost every $f$,   i.e.~$\nu(f(\Omega)\Delta\Omega)=0$ for $\mu$-a.e.~$f$. 
\end{enumerate}
\end{proposition}

\subsection{Transversality.}
We recall here Thom's Jet Transversality Theorem. See for example \cite{golubitsky1973stable} or \cite{cieliebak2024introduction} for a general discussion.

Let $X$ and $Y$ be smooth manifolds and  $W$ be a submanifold in $Y$. We say that a smooth map 
$f:X\to Y$ is transversal to $W$ if for each $x\in X$ such that
$f(x)\in W$ we have that $T_{f(x)}Y=T_{f(x)} W+ Df (T_x X).$ We will use the notation $f\pitchfork W$ to mean that $f$ is transversal to $W.$
Note that if $f\pitchfork W$ and $\dim(X)+\dim(W)<\dim(Y)$ then the image of $X$ is disjoint from $W.$
We also recall that for each smooth map $f$ from $X$ to $Y$ and each $k$ there is a smooth map $j^k f$ from $X$ to the space $J^k (X,Y)$ of $k$-jets.
The following result is helpful for constructing maps with certain properties:

 \begin{theorem}
 \label{ThThom}
 \cite[Thm.~4.9]{golubitsky1973stable}
 (Thom Jet Transversality Theorem) Let $X$ and $Y$ be smooth manifolds and $W$ be a submanifold of $J^k(X,Y)$. Then 
 \[
T_W=\{f\in C^{\infty}(X,Y)\mid j^kf\pitchfork W\}
 \]
 is a residual subset of $C^{\infty}(X,Y)$ in the $C^{\infty}$ topology. 
 \end{theorem}

 We emphasize that the submanifold $W$ in this theorem need not be closed or compact.
The Thom transversality theorem also applies in the volume preserving setting \cite[Thm.~3]{visik1971some}. 

\subsection{Measure Theory}
 The following result is useful in proofing that certain properties are generic. 
\begin{proposition}
\label{PrNoInvSetsDiff}
Let $A$ be a measurable  set in a closed manifold $M$ such that $0<\vol(A)<\vol(M)$. Then for each $r\geq 1$ the set
$$\cN_r(A)=\{g\in \Diff^r_{\vol} (M): \vol(gA\cap (M\setminus A))>0\}$$ is open and dense.
\end{proposition}

\begin{proof}
Denote $B=M\setminus A$ and let $A^*$ and $B^*$ be the density points of $A$ and $B$ respectively.   

To see that $\cN_r(A)$ is open, take $g\in \cN_r(A).$ Since 
$\vol(A\Delta A^*)=\vol(B\Delta B^*)=0$, we have that $g A^*\cap B^*\neq\emptyset.$
Take $x\in A^*$ such that $y=g(x)\in B^*$. Since $g$ is Lipshitz, there is a constant $\delta>0$ such that 
for all $t$ small enough the sets $U_1=B(x, t)$, $U_2=B(y, \delta t)$ satisfy that:
\begin{enumerate}
    \item[(i)]
$U_1$ and $U_2$ are closed;
    \item[(ii)]
    $U_2 \subset \mathrm{Int}(g U_1)$; 
    \item[(iii)]
    $\vol(B\cap U_2)+\vol(A\cap U_1)>\vol(U_1)=\vol(g U_1).$
\end{enumerate}
For a fixed sufficiently small $t>0$, (i)--(iii) will also be satisfied with $g$ replaced by its small perturbation $\tg$, which shows that
$\tg\in \cN_r(A)$, whence $\cN_r(A)$ is open.

To show that $\cN_r(A)$ is dense we need to show that any diffeomorphism $g$ can be approximated by diffeomorphisms from
$\cN_r(A).$ If $g\in\cN_r(A)$ we are done, so we may assume that $\vol(A\Delta gA)=0.$ Then $g$ preserves
$A^*$. Let $z$ be a point on the boundary of $A^*$. Then for each $r$, $B(z,r)$ contains points
from both $A^*$ (since $z\in\partial A^*$) and from $B^*$ (since otherwise $z\in\mathrm{Int}(A^*))$. 
Thus there are points $x_n\in A^*, y_n\in B^*$ converging to $z$. 
Hence there are maps $h_n$ arbitrary close to identity such that $h_n x_n=y_n$ and hence
$h_n A^*\cap B^*$ is non-empty. 
Then $\tg_n=h_n\circ g$ also has this property. Now the same argument as in first part of the proof shows that
$\tg_n\in \cN_r(A)$. Since $\tg_n\to g$, $\cN_r(A)$ is dense.
\end{proof}

\section{Expanding on average conditions}\label{sec:expanding_on_average_conditions}
\subsection{Bundle maps associated to a random system}
\label{SSBundleDef}
In order to adequately describe the expanding on average conditions that we use, we introduce a small amount of formalism. 
Suppose that $\mc{E}$ is a Riemannian vector bundle over a smooth manifold $M$. Let $\Aut(\mc{E})$ be the space of all vector bundle automorphisms of $\mc{E}$ fibering over a homeomorphism of $M$, and let $\Aut^{\infty}(\mc{E})$ be the space of all $C^{\infty}$ bundle automorphisms of $\mc{E}$ fibering over 
$C^{\infty}$ diffeomorphisms of $M$. For example, for any $C^{\infty}$ diffeomorphism $f$, $Df\in \Aut^{\infty}(TM)$. Now consider a measure $\mu$ supported on the space of maps $F\colon \mc{E}\to \mc{E}$ in $\Aut^{\infty}(\mc{E})$.

\begin{definition}
We say that a measure $\mu$ on $\Aut(\mc{E})$ is \emph{expanding on average} if there exists $N,\lambda>0$ such that for every unit vector $v\in \mc{E}$, 
\begin{equation}
\int \ln\|F^N_{\omega} v\|\,d\mu^N(\omega)> \lambda>0.
\end{equation}
\end{definition} 

There is also a more general notion of expanding on average on $k$-planes, which seems to first be mentioned in \cite[Def.~1.2]{elliott2023uniformly}. 

\begin{definition}
Suppose that $\mu$ is a probability measure on $\Aut(\mc{E})$. Then we say that $\mu$ is expanding on average on $k$-planes if the following holds. There exists $N,\lambda>0$, such that for all $k$-planes $V$ in $\mc{E}$,
\[
\int \ln \|F^n_{\omega} \vert_{\vol_V}\|\,d\mu^n(\omega)>\lambda>0. 
\]
\end{definition}

Note that given dynamics in $\Aut(\mc{E})$ there are naturally associated random bundle maps of the associated Grassmannian bundles. For a measure $\mu$ we let $\mu_k$ denote the associated random dynamics on $\Gr_k(\mc{E})$. If the dynamics of $\mu$ are denoted $F$, then we write $F_k$ for the induced dynamics of $F$ on $\Gr_k(\mc{E})$.

\subsection{Characterization of Expansion on Average}

The expanding on average property for bundle automorphisms is characterized similarly to the expanding on average property for diffeomorphisms. The proof of the following is a straightforward extension of \cite[Thm.~3.2]{elliott2023uniformly}, which is a generalization of the proof of \cite[Prop.~3.17]{chung2020stationary}, although \cite[Thm.~3.2]{elliott2023uniformly} does not claim the full characterization that \cite{chung2020stationary} obtains. See also the discussion in \cite{cantat2025hyperbolicity}.
\begin{proposition}\label{prop:EoA_int_characterization}
Let $\mc{E}$ be a smooth Riemannian vector bundle and suppose that $\mu$ is a probability measure on $\Aut(\mc{E})$ with bounded support. Then $\mu$ is expanding on average if and only if for all a $\mu$-stationary measures $\nu$ on $\PP(\mc{E})$,
\begin{equation}
\label{eqn:1Planes}
\iint \ln \|Df v\|\,d\nu(v)\,d\mu(f)> 0.
\end{equation}

The analogous characterization holds for the expansion on average on $k$-planes. 
 Namely, $\mu$ is expanding on average on $k$-planes if and only if for all $\mu_k$ stationary measures $\nu$ on $\Gr_k(\mc{E})$, 
\begin{equation}
\label{eqn:KPlanes}
    \iint \ln \|F_k\vert V\|\,d\nu(V)\,d\mu(F_k)>0.
\end{equation}
\end{proposition}
\begin{proof}
First suppose that $\mu$ is not expanding on average. Then for every $n\in \N$, there exists a vector 
$v_n$ such that
\[
\int \ln \|F^n_{\omega}({v_n})\|\,d\mu^n(\omega)\le 0.
\]
Let $\nu_n$ be the measure
$\int \delta_{F^n_{\omega}v_n}\,d\mu^n(\omega),$ 
and $\overline{\nu}$ be a weak* limit of the measures 
$\nu'_n={\frac{1}{n}}\sum_{i=0}^{n-1} \nu_i$. As $\mu*\nu'_n$ is increasingly close to $\nu'_n$, it follows that $\overline{\nu}$ is $\mu$-stationary. 
Further, for any $\epsilon$ and all large $N$ we have that $\int \ln \|F_{\omega}^nv_n\|\,d\mu^n(\omega)\le \epsilon$. Hence, for any weak* limit we have by continuity that 
$\displaystyle
\int \ln \|F_{\omega}v\|\,d\overline{\nu}(v)\le \epsilon.
$
But $\epsilon>0$ was arbitrary so we obtain the needed conclusion. 

Suppose now that $\mu$ is expanding on average; then it is straightforward to see that there exists $\lambda>0$ such that not only is the integral of $\nu$ positive, but in fact for any stationary measure $\nu$, 
the integral is at least $\lambda$.  This completes the proof.
The argument in the case of $k$-planes is identical.
\end{proof}

In the following proof we say that a vector $v$ is almost surely non-expanding if almost surely $\limsup \frac{\ln \|F^n_{\omega}v\|}{n}\le 0$. Also a {\em $\nu$-measurable family of subbundles} is a collection of $k$-dimensional subspaces in $\R^d$ defined at $\nu$-a.e. point. The invariance of such a measurable family means that this collection is permuted by the random dynamics $\mu$. 

\begin{proposition}\label{prop:EoA_characterization}
Suppose $\mc{E}$ is a Riemannian vector bundle over a smooth manifold $M$ and that $\mu$ is a probability measure on $\Aut(\mc{E})$. 
\begin{enumerate}[leftmargin=*]
\item[(a)] 
The measure $\mu$ is  expanding on average if and only if for every stationary measure $\nu$ on $M$, there is no non-trivial $\nu$-measurable $\mu$-a.s.~invariant subbundle of $\mc{E}$ comprised of vectors that a.s.~have Lyapunov exponent at most $0$.

\item[(b)]
A volume preserving driving measure $\mu$ is expanding on average if for all stationary measures $\nu$ there do not exist  any $\mu$-a.s.~invariant $\nu$-measurable family of subbundles or a $\nu$-measurable 
 Riemannian metric.
\end{enumerate}
\end{proposition}

 The advantage of part (b) is that in the volume preserving case we do not need to verify 
expansion directly, only rule out measurable invariant structures.

Also, note that the condition (b)  is not necessary  for expansion on average. A simple counterexample 
 is the product of two expanding on average systems. That is, if $\mu$ is expanding on average, then the measure 
 $\mu\times \mu$ on $\Diff_{\vol}^\infty (M\times M)$ is expanding on average, see \S \ref{subsec:products} below for this and similar examples. 

Before we proceed, we comment on condition (b), which is slightly different than the statements appearing in the literature. For example, a similar characterization appears in \cite[Thm.~1.2]{potrie2022remark}, but without the added statement that there might be more than one subbundle. Its conclusion reads, that if the dynamics is not expanding on average, ``[...] then there is an invariant $\nu$-measurable distribution or conformal structure\footnote{Note that for volume preserving linear cocycles, having a measurable invariant Riemannian metric is the same thing as having a measurable conformal structure.}
."  
The corresponding statement in \cite[Lem.~3.2]{elliott2023uniformly} gives a less precise type of characterization, which says that there are no $\nu$-measurable algebraic structures in $\Gr_k(TM)$. Such structures can have more than one connected component. However, compare the statement of \cite[Lem.~3.5]{elliott2023uniformly} with the last line of that lemma's proof to see that a similar issue appears. 
As we will shortly explain, systems that are expanding on average but do not have an invariant measurable subbundle or Riemannian metric do occur. That said, the main results of the papers just mentioned are certainly unaffected: these are just minor oversights and do not affect the strategy of the proofs,  because the methods
used in \cite{potrie2022remark, elliott2023uniformly} to rule out invariant bundles also allow to rule out  
families of such bundles (cf. Lemma \ref{LmSkewDensity} in the present paper). 

Here is an example of non-expanding on average dynamics without an a.s.~invariant line bundle or Riemannian metric 
over a stationary measure. Suppose that $(f_1,f_2)$ are two volume preserving diffeomorphisms of a closed surface $M$, and that $p$ is a common fixed point where their differentials are the matrices:
\begin{equation}
\label{MultRot}
    \begin{bmatrix}
\lambda & 0 \\
0 & \lambda^{-1} 
\end{bmatrix}\text{ and } 
\begin{bmatrix}
0 & -1\\
1 & 0
\end{bmatrix}.
\end{equation}
Then $\nu=\delta_p$ is an invariant measure  for the driving measure $2^{-1}(\delta_{f_1}+\delta_{f_2})$
where all Lyapunov exponents at $p$ vanish. However there is no invariant Riemannian metric at $p$ nor a line bundle. On the other hand, the union of the $x$ and $y$ axes is certainly invariant. 
Moreover, applying the techniques used in \cite{elliott2023uniformly} and the proof of 
Theorem \ref{ThGenVF-EoA} below, 
 one can produce an example of a random measure where all maps preserve $p$, the derivative at  $p$ is given by
\eqref{MultRot} above, the only  stationary measures are $\delta_p$ and the volume,
and there are no invariant structures over volume either.

\begin{proof}[Proof of Proposition~\ref{prop:EoA_characterization}.]
The necessity of (a) is obvious, so we only show the other direction. 
From Proposition \ref{prop:EoA_int_characterization}, it follows that there exists an ergodic stationary measure $\hat{\nu}$ on $\mc{E}$ such that $\int \|F_{\omega}v\|\,d\nu(v)\,d\mu(\omega)\le 0$. We let $\nu$ denote the pushforward to the base.

Given $x\in M$, consider the subset of $\PP(T_xM)$ of vectors $v$ such that $v$ is almost surely non-expanding, i.e.~the Lyapunov exponent of the vector $v$ is non-positive. Note that if $v,w\in T_xM$ are almost surely non-expanding, then so is every vector in their span. 
Thus we see that there is a well defined a.s.~non-expanding subspace over each point $x\in M$, which we call $V_{non}(x)$. 
In fact, note that for $\hat{\nu}$-a.e.~$v\in \mc{E}$, by the Birkhoff ergodic theorem $v$ is a.s.~non-expanding. 
Thus we see that over $\nu$-this subspace is nontrivial. 
Further note that any $w\notin V_{non}(x)$ is \emph{not} almost surely non-expanding, i.e.~a.s.~$\liminf n^{-1}\ln \|F^n_{\omega} w\|>0$. Since $\nu$ is ergodic, the dimension of $V_{non}(x)$ is a.s.~constant, call this dimension $k$. Note that due to its characterization $V_{non}$ is a.s.~invariant. This finishes the proof of the characterization in the non-conservative case.

We now prove the alternative criterion for the conservative case.
If the subspaces $V_{non}$ we found in the above part had  dimension $k<d$, then we are done.
So suppose that $k=d$, we then need to produce an invariant subbundle family or Riemannian 
metric.

We now apply the invariance principle to upgrade $\hat{\nu}$ from a stationary measure to an invariant measure.
Let $\Sigma$ be the space $\Diff(M)^\N\!\times\!\! \Gr_k(\mc{E})$ endowed with the measure $\mu^N\!\!\times\! \hat{\nu}$. 
 By \cite[Thm.~B]{AvilaViana10}, 
the disintegration of $\hat{\nu}$ along fibers depends only on the zeroth symbol. But this implies that the disintegration of $\hat{\nu}$ is a.s.~invariant under all of the dynamics. We have a map that sends $\omega_0\omega_1\in \Diff(M)$ to the disintegration $\hat{\nu}_{\omega_0,\omega_1}$, and as the disintegration of the image of this vector is invariant we have that $f_{\omega_0}\hat{\nu}_{\omega_0}\!\!=\!\!\hat{\nu}_{\omega_1}$ for a.e.~$\omega_1$. Thus the disintegration is almost surely equal to some constant $\hat{\nu}$; this is a measure on $\PP(\mc{E})$ that is a.s.~invariant 
by~$\mu$.

As before, we may assume that $\hat{\nu}$ is ergodic. 
By \cite[Lem.~3.22]{arnold1999jordan}, if the cocycle does not preserve  a measurable Riemannian metric, then for almost every $\omega$, the conditional measure $\hat{\nu}_{\omega}=\hat{\nu}$ is supported on the union of two proper subspaces $[V]$ and $[W]$. 
 In this case we will produce a finite collection of subspaces that are permuted. 

First, if there exist any atoms of $\nu$ in $\Gr_1(TM)$, then we are done, because the atoms of a fixed mass are an almost surely invariant set. So, suppose there are no atoms of the disintegration of $\nu$ in $\Gr_1(TM)$. Then there is a minimum $k<d$ such that $\nu$ assigns positive measure to some $k$-dimensional subspace. Note that $k<d$ due to the support of the disintegration of $\nu$ being contained in the union of two subspaces $[V]\cup [W]$ from the previous paragraph. Then due to ergodicity there exists some $0<\eta<1$, such that at $\nu$-a.e. point there is a plane $V$ whose measure is $\eta$. Note that there are at most finitely many such planes in each fiber as their intersection is a set of zero measure. Hence at each point we have a finite collection $V_1(\omega),\ldots,V_k(\omega)$ for some a.s.~constant $k$. Further, note that there is some maximum $\eta$ such that the foregoing statement is true as each fiber has mass $1$. But this implies that the set of such mass $\eta$ planes over each point must be a.s.~invariant because otherwise stationarity would be violated: Every preimage of such a plane must be a plane of at least measure $\eta$. Thus this collection of planes is a $\nu$-measurable $\mu$-a.s.~invariant finite collection of subspaces. We have obtained the needed dichotomy.
\end{proof}

\begin{remark}\label{rem:two_subbundles}
Note that the above proof furnishes additional information in dimension $2$: any invariant family of line bundles is supported on at most two lines at each point. Otherwise the fact that the disintegration of $\nu$ is supported on two non-trivial subspaces in the penultimate paragraph would not hold. 
\end{remark}

\begin{definition}
 \label{DefClean}   
 We say that a measure $\mu$ on $\Diff(M)$ is {\em clean} if:
\begin{enumerate}[leftmargin=*]
    \item[(i)]
For each $x\in M$ the distribution of $fx$ has an absolutely continuous component; 
\item[(ii)]
volume is ergodic for $\mu$; 
 \item[(iii)]  there do not exist any  measurable $\mu$ a.s.~invariant family of line bundles  or a 
 $\mu$-a.s.~invariant Riemannian metric on  $M$.
 \end{enumerate} 
  Here we say that a measurable  Riemannian metric $g$ is $\mu$ a.s.~invariant if for $\mu$ almost every $f$:
\[
\DS \vol(x: \forall v\in T_x M\,\,\, g_x(v)=g_{fx}(D_xf v))=1.
\]
A $\mu$ a.s.~invariant family of line bundles is defined similarly.
 
 We say that a measure $\mu$ on $\Diff(M)$ is \emph{coclean} if (i) and (ii) along with the condition $(iii)'$ below hold: 
 
\begin{enumerate}[leftmargin=*]
    \item[$(iii)'$] 
For the induced action of $\mu$ on $T^*M$, there do not exist any  measurable $\mu$~a.s. invariant family of subbundles of $T^*M$  or a 
 $\mu$~a.s. invariant Riemannian metric on  $T^*M$.
\end{enumerate}

 \end{definition}

 \begin{corollary}
 \label{rem:clean}
 If
 $\mu$ is clean and $\tilde \mu$ is another measure on $\Diff(M)$ such that $\mu$ is absolutely continuous with respect to $\tilde\mu$, then
 $\tilde\mu$ is expanding on average and ergodic. 
 In particular if $\bar\mu$ is an arbitrary measure then for each $\eps\in (0, 1]$
 the measure $\eps\mu+(1-\eps)\bar\mu$ is expanding on average and ergodic. The same holds for the conclusion that $\mu$ is coexpanding on average, if $\mu$ is assumed to be coclean.
\end{corollary}

\begin{proof}
We will check only the first claim about clean measures; the proof for coclean measures is identical.

First we show that that the volume is the unique $\tilde\mu$ stationary measure. Indeed let $\nu$ be an ergodic 
stationary measure. By (i) $\nu$ has an absolutely continuous component and since the class of 
absolutely continuous measures is invariant under convolution with $\mu$, $\nu$ must be absolutely continuous as all its mass must belong to its absolutely continuous component. 
By (ii) and Proposition \ref{PrRandomET},
 every $\mu$ almost surely invariant subset of $M$ has null or conull volume.
Applying Proposition \ref{PrRandomET} again we see that $\nu$ is volume.

However by (iii) there are no $\mu$ invariant and, hence, $\tilde\mu$ invariant, geometric structures and
so by Proposition \ref{prop:EoA_characterization}, $\tilde\mu$ is expanding on average.
\end{proof}

Given a measure $\mu$ on $\Diff^{\infty}_{\vol}(M)$ we have four different associated bundle maps. Write $Df\colon TM\to TM$ for the derivative. Write $Df^*\colon T^*M\to T^*M$ for the pullback, which maps fibers $T^*_{x}M\to T^*_{f^{-1}(x)}M$. 

Associated to the measure $\mu$ there are four basic associated random bundle automorphisms that one might study. In square brackets, we give them a name corresponding to their relationship with the original maps $f$. We list them as a pair $(f,F)$, where $f$ is a diffeomorphism and $F$ is a bundle map covering $f$.
\begin{definition}\label{defn:different_types_of_cocycles}
For a measure $\mu$ on $\Diff^{\infty}(M)$, we have four associated random bundle maps, and refer to the condition of each of them being expanding on average as follows:
\begin{enumerate}
\item
$(f,Df)$ on $TM$. [expanding on average]
\item
$(f,(Df^*)^{-1})$ on $T^*M$. [coexpanding on average]
\item
$(f^{-1}, Df^{-1})$ on $TM$. [expanding on average backwards]
\item
$(f^{-1}, Df^*)$ on $T^*M$.  [coexpanding on average backwards]
\end{enumerate}

\end{definition}

In fact, one can define the same four notions for any collection of bundle maps. We won't bother constructing diffeomorphisms that show each of these classes is distinct, but for bundle maps it is quite easy. In fact, we can do it with dynamics over a single point.

\begin{proposition}
Let $\mc{E}=\{*\}\times \R^3$, where $\{*\}$ is the singleton topological space. 
For any $1\le i \le 4$, one of the four different type of expanding on average (1)--(4) in the above list, there is a probability measure $\mu$ with bounded support on $\Aut(\mc{E})$ such that $\mu$ is not expanding on average of type ($i$), but is expanding on average of the other three types. For example, there is a measure that satisfies $(2),(3),(4)$, but not $(1)$.
\end{proposition}
\begin{proof}
This is straightforward using the characterization in Proposition~\ref{prop:EoA_characterization}.
We give an example of a measure $\mu$ that is not expanding on average for (1) but is for each of the others. The other examples we obtain the other cases, one can replace $\mu$ by the measures $\mu^{-T}, \mu^{-1}$, and $\mu^{T}$.
 Here, by $\mu^{-T}$ we mean the pushforward of $\mu$ by the map $A\mapsto A^{-T}$; the others are defined analogously.

Let $\mu_B$ be a measure supported on $\GL(2,\R)$. We will take $\mu_B$ to be a measure that is expanding on average and such that $\mu_{\det B}$ is also uniformly expanding. We may also choose $B$ so that the measures $\mu_{B^{-1}},\mu_{B^{-T}}$ and $\mu_{B^{-T}}$ are also expanding on average
because all those conditions are generic (see 
Example \ref{ExTorus} in \S \ref{ScHomogeneous} for a detailed discussion). In fact, by taking a power of these measures, we may arrange that all four of these measures are expanding on average with $N=1$ and that for all unit vectors $v$, if $\mu'$ is one of these four measures,
\begin{equation}\label{eqn:expansion_assump1}
\mathbb{E}_{\mu}{\ln \|F_{\omega} v\|}>100.
\end{equation}

Now consider the automorphisms of the trivial bundle $\{*\}\times \R^3$ distributed according to the measure $\mu_C$, where $C$ is distributed according to:
\[
C=\begin{bmatrix}
\det(B)^{-1} & q\\
0 & B
\end{bmatrix},
\]
where $q\in \R^2$ is a vector of length $L$, which is a constant that we will choose later.

Then four associated random walks arise distributed according to the measures $\mu_C,\mu_{C^{-T}},\mu_{C^{-1}},\mu_{C^T}$: 
\[
\begin{bmatrix}
\det(B)^{-1} & q\\
0 & B
\end{bmatrix}, 
\begin{bmatrix}
\det(B) & 0\\
* & B^{-T}
\end{bmatrix},
\begin{bmatrix}
\det(B) & *\\
0 & B^{-1}
\end{bmatrix},
\begin{bmatrix}
\det(B)^{-1} & 0\\
q^T & B^{T}
\end{bmatrix},
\]
where we have written a $*$ for whatever that entry must be. 

We now explain why $\mu_C$ is not expanding on average but the rest are.

(1)
For $\mu_C$, as the first coordinate is contracting, it is clear that $\mu_{C}$ is not expanding on average as it has an almost surely contracting subbundle.

(2) For $\mu_{C^{-T}}$, both of the blocks on the diagonal elements are expanding on average with $N=1$, hence any vector will expand in one step.

(3) For $\mu_{C^{-1}}$, similarly both of the diagonal blocks are expanding on average in one step, hence so is $\mu_{C^{-1}}$.

 (4)
For $\mu_{C^{T}}$, we need to argue slightly more as now the first block does not expand. First note that any unit vector that lies in the subspace $\{0\}\times \R^2$ will certainly expand due to \eqref{eqn:expansion_assump1}. In fact, by continuity, we see that the same holds for all unit vectors that make an angle of at most $\epsilon_0$ with $\{0\}\times \R^2$. But any vector $v$ that makes angle at least $\epsilon_0$ with $\R^2$ has first component at least $\epsilon_0/2$ in magnitude, hence for any matrix $C$ in the support of $\mu_{C^{T}}$, $\|Cv\|\ge L\epsilon_0$. Thus for $L$ sufficiently large, we see that this measure is expanding on average as well.
\end{proof}

We can also give constructions in the case of diffeomorphisms.
\begin{proposition}\label{prop:coexpanding_but_not_expanding}
There exists a measure $\mu$ with compact support on $\Diff^{\infty}_{\vol}(\mathbb{T}^4)$ that is coexpanding on average but is not expanding on average.
\end{proposition}
\begin{proof}
To begin, let $\mu_A$ be a measure with finite support on $\SL(2,\Z)$ that is expanding on average at time $N=1$ and satisfies all four of the types of expanding on average in Definition \ref{defn:different_types_of_cocycles} with a uniform lowerbound $M>0$ on the expansion in each case. Note that by taking a convolution $\mu^n_A$ we can make $M$ as large as we like.

We now define two measures $\hat{\mu}_A$ and $\hat{\mu}_L$ that are both supported on $\SL(4,\Z)$. We define $\hat{\mu}_A$ to be a pushforward of $\mu_A$ by the map 
\[
A\mapsto \begin{bmatrix}
A & 0\\
0 & \Id_2
\end{bmatrix},
\]
and $\hat{\mu}_L$ will be supported on the constant shear matrix 
\[
\begin{bmatrix}
\Id_2 &  0\\
L \Id_2 & \Id_2.
\end{bmatrix}
\]
We then claim that for suitably chosen $M,L$ that $\hat{\mu}=(\hat{\mu}_A+\hat{\mu}_L)/2$ is coexpanding on average but not expanding on average. 
The corresponding cocycle on $T^*\mathbb{T}^4$ takes the form:
\[
\begin{bmatrix}
A^{-T} & 0\\
0 & \Id_2
\end{bmatrix}
\,\, , \,\,
\begin{bmatrix}
\Id & -L\Id\\
0 & -\Id
\end{bmatrix}
\]
We claim that for $L=10$ that if $M$, the expansion on average constant, is sufficiently large, then this random matrix product is coexpanding on average. 

We will check the definition of coexpanding on average directly. Suppose that $v=(x,y)\in \R^2\oplus \R^2$ is a unit vector. Then there are two cases depending on where $v$ lies. Fix $\epsilon=1/100$.

(1) ($\|x\|\le \epsilon$) In this case $\|y\|\ge 1-\epsilon^2$. Thus we can compute, that 
    \begin{equation}\label{eqn:average_of_measures_expectation}
    2\mathbb{E}_{\hat{\mu}}[\ln \|Bv\|]=\mathbb{E}_{\hat{\mu}_A}[\ln \|Bv\|]+\mathbb{E}_{\hat{\mu}_L} [\ln \|Bv\| ]=(i)+(ii).
    \end{equation}
    Due to the diagonal structure of the matrix, we obtain the trivial bound
    \[
    (i)=\mathbb{E}_{\hat{\mu}_A}\left[\ln \left\|\begin{bmatrix}
        A^{-T} & 0 \\
        0 & \Id 
        \end{bmatrix}
        \begin{bmatrix}
            x \\ y
            \end{bmatrix}
            \right\|\right]\ge \ln(1-\epsilon^2).
    \]
   Also
    \begin{align*}
    (ii)=\ln \left\| \begin{bmatrix}
\Id & -L\Id\\
0 & -\Id
\end{bmatrix}
\begin{bmatrix}
    x \\
    y
\end{bmatrix}
\right\|&\ge\ln \left( \left\|\begin{bmatrix}
\Id & -L\Id\\
0 & -\Id
\end{bmatrix}
\begin{bmatrix}
    0 \\
    y
\end{bmatrix}\right\|- \left\|\begin{bmatrix}
\Id & -L\Id\\
0 & -\Id
\end{bmatrix}
\begin{bmatrix}
    x \\
    0
\end{bmatrix}\right\|\right)\\
&\ge \ln (L(1-\epsilon^2)-\epsilon)>0
    \end{align*}
    Thus we see that the expansion on average condition is satisfied for vectors with $\|x\|\le \epsilon$, given $L$ and our choice of $\epsilon$.
    
    (2) ($\|x\|\ge \epsilon$) In this case we will take advantage of the expansion on average condition. As in the previous case, we have decomposition according to equation \eqref{eqn:average_of_measures_expectation} into two terms $(i)$ and $(ii)$. 
    \[
    (i)=\mathbb{E}_{\hat{\mu}_A}\left[\ln \left\|\begin{bmatrix}
        A^{-T} & 0 \\
        0 & \Id 
        \end{bmatrix}
        \begin{bmatrix}
            x \\ y
            \end{bmatrix}
            \right\|\right]\ge \mathbb{E}_{\hat{\mu}_A}\left[\ln \left\|\begin{bmatrix}
        A^{-T} & 0 \\
        0 & \Id 
        \end{bmatrix}
        \begin{bmatrix}
            x \\ 0
            \end{bmatrix}
            \right\|\right]\ge M-\ln \epsilon.
    \]
Also
\[
(ii)=\ln \left\| \begin{bmatrix}
\Id & -L\Id\\
0 & -\Id
\end{bmatrix}
\begin{bmatrix}
    x \\
    y
\end{bmatrix}
\right\|\ge \ln \sigma_4\ge \ln 2L, 
\]
where $\sigma_4$ is the smallest singular value of this matrix. 
Thus as long as 
\[
M-\ln \epsilon -\ln 2L>0,
\]
the random dynamics are expanding on average for these vectors as well.

Thus given our choice of $L=10$ and $\epsilon=1/100$, as long as 
\[
M\ge \ln(1/100)+\ln 20,
\]
the measure $\hat{\mu}$ is coexpanding on average.

 Noting that this system cannot be expanding on average because the last $2$ coordinates do not grow under the dynamics completes the proof.
\end{proof}

\begin{example}
 Note that in the above proof if we had reversed the roles of $\mu_A$ and $\mu_{A^{T}}$, then the dynamics on $\mathbb{T}^4$ would be generated by matrices of the form 
\[
\begin{bmatrix}
A & 0\\
0 &\Id
\end{bmatrix}, 
\begin{bmatrix}
\Id & L\Id\\
0 & \Id
\end{bmatrix}.
\]
The same argument as above shows that the random walk of these matrices will be expanding on average. However, note the random dynamics generated by these matrices is not ergodic because all maps factor over the identity map on $\mathbb{T}^2$. Thus $x_3$ and $x_4$ are continuous invariant functions for our system. 
This is especially striking in view of Corollary \ref{cor:coexpanding_components} below, which shows that under the coexpanding on average condition, there would only have been finitely many totally ergodic components of volume
(Note that, by e.g. Proposition \ref{PrRandomET} ergodic components of $\mu$ and $\mu^{-1}$ are the same).
\end{example}

Below we will need the following alternative characterization of the coexpanding on average condition. 

\begin{proposition}\label{prop:coa_d-1_planes}
Suppose that $\mc{E}$ is a $d$-dimensional vector bundle over a topological space $X$. 
Suppose that $\mu$ is an measure on $\Aut_{\vol}(\mc{E})$. Then $\mu$ is coexpanding on average if and only if it is expanding on average on $d-1$ planes.
\end{proposition}
\begin{proof}
Fix a Riemannian metric on $\mc{E}$ so that the induced volume form of the metric agrees with the volume form already on $\mc{E}$. (As all volume forms are proportional, any metric will have this property after rescaling.)
First we begin with an observation. Suppose that $V$ is a $(d-1)$-plane in $\mc{E}$ and that $L\colon \mc{E}_x\to \mc{E}_y$. Then we can fix orthonormal frames $(n_V,v_1,\ldots,v_{d-1})$ and $(n_{L(V)},v_1',\ldots,v_{d-1}')$ such that $n_V$ and $n_{L(V)}$ are orthogonal to $V$ and $L(V)$, and the $v_i$ are an orthonormal basis of $V$ and the $v_i'$ are an orthonormal basis of $L(V)$. Then with respect to this ordered basis the $L$ is represented by a matrix:
\[
\begin{bmatrix}
a & 0\\
b & C
\end{bmatrix},
\]
where $a\in \R$, $b\in \R^{d-1}$, and $C\in GL(d-1,\R)$. This is the action $\mc{E}_x\to \mc{E}_y$. 
Then $L_*\colon \mc{E}_x^*\to \mc{E}_y^*$ is given by the matrix 
\[
\begin{bmatrix}
a^{-1} & b'\\
0 & C^{-T}
\end{bmatrix}.
\]
Let $n_V^*$ denote the dual vector to $n_V$. Then as $a\det(C)=1$ due to volume preservation, we see that 
\[
\|F\vert{n_V^*}\|=\|F\vert_{\vol(V)}\|,
\]
i.e.~the norm of the action on the conormal to $V$ is the same as the action on the volume element of $V$.

Thus the coexpansion on average is the same thing as being expanding on $(d-1)$-planes because
$\displaystyle
\int_{\omega} \ln \|F^n_{\omega} n_V^*\|\,d\mu^n_{\omega}=\int_{\omega} \ln \|F^n_{\omega}\vert_{\vol(V)} \|\,d\mu^n_{\omega}
$.
\end{proof}

 An immediate consequence of the above result is that for volume preserving systems in dimension $2$, expanding on average and coexpanding on average are the same thing. Both are equivalent to being expanding on average on lines ($1$-planes). 

\begin{corollary}
    \label{prop:low_dim_equivalence}
Suppose that $\mc{E}$ is a two dimensional vector bundle over a manifold $M$. If $\mu$ is a measure with compact support on $\Aut(\mc{E},M)$ such that the induced bundle automorphism on $\Lambda^2\mc{E}$ preserves a non-vanishing volume, then $\mu$ is expanding on average if and only if it is coexpanding on average, i.e. the corresponding measure $\mu^*$ on $\Aut(\mc{E}^*,M)$ over the same base dynamics is expanding on average.
\end{corollary}

\section{Examples}\label{sec:examples}
 It was proven in 
 Potrie \cite{potrie2022remark} (for surfaces) and Elliot-Smith \cite{elliott2023uniformly} (in arbitrary dimension)
 that the set of conservative measures which are ergodic and expanding on average on $k$ planes is weakly dense, for every $k$.
 By Corollary \ref{prop:low_dim_equivalence} the set of ergodic measures which are coexpanding on average is also dense.
 As the coexpanding on average property is also manifestly $C^1$ open, this shows that ergodic coexpanding on average measures are  weak$^*$ generic. For many examples of generators $\mu$ in this section, we will focus on the coexpanding and expanding on average conditions, rather than the corresponding conditions on $\mu^{-1}$, which follows similarly in the cases below, but is slightly less natural to think about. 

 In this section we discuss several specific models of random dynamics studied in the literature and show that
 many are coexpanding on average.

\subsection{Random flows}
 \label{SSContExamples}

 Here we show how to verify coexpanding on average condition for measures of large support.
Our arguments are close to constructions of 
 \cite{BH12, BBMZ, BCZG, elliott2023uniformly, potrie2022remark} 
 but we provide details, since the model considered below is of independent interest.
 We note that for most of the examples of \S \ref{SSContExamples} it seems possible to verify the  stronger assumptions of
 \cite{BFPS} (in fact for Example \ref{ExPierrehumbert} this is done in \cite{BCZG, BFPS}). However,
as it was mentioned in the introduction, the advantage of our assumptions is that they are stable under 
weak* small perturbations, and so they remain valid if $\mu$ is approximated, for example, by atomic measures.

\begin{example}  Take $p>1$ and let $\cX=(X_1, \dots, X_p)$ be a tuple of  smooth 
divergence free vector fields on $M$, a closed manifold of dimension at least $2$.
Denote by $\Phi_j(t)$ the time $t$ map generated by the flow of $X_j$.
Let $d=\dim(M).$ 
Fix $T>0$, and let $(t, j)$ be uniformly distributed on $[-T, T]\times \{1, \dots, p\}$,
and let $\mu=\mu_\cX$ be the law of $\Phi_j(t)$.
\end{example}

\begin{theorem}
\label{ThGenVF-EoA}
For any fixed $T$ and $p\geq 2$, $\mu_\cX$ is coexpanding on average for an  open and dense  set of tuples $\cX$.
\end{theorem}

We need some preparations.
We say that $\cX=(X_1, \dots ,X_p)$ has the {\em accessibility property} if for each $x$ and $y$ in $M$ there are $r\in \mathbb{N}$,
$\vec{j}=(j_1, \dots, j_r)$ and $\vec{t}=(t_1, \dots ,t_r)$ such that 
$\DS  \Phi_{\vec j}(\vec t)x:=\Phi_{j_r}(t_r) \cdots \Phi_{j_1}(t_1) x=y.$ Note that for fixed $\vec j$, 
we can view $\Phi$ as a smooth map $\Phi_{\vec{j}}\colon \R^r\to M$. We denote the set of 
tuples with the accessibility property by $\fA.$
The following results can be found 
in \cite[Section 3]{PSComplexity}.

\begin{theorem}
\label{Th-PuSh}
Suppose that $M$ is a smooth manifold. Then we have the following properties of accessible vector fields.:
\begin{enumerate}[leftmargin=*]
    \item[(a)] \label{ThControl} \cite[Thm.~3.2]{PSComplexity}
 If $\cX\in \fA$ then for each $x, y\in M$ there exist  $\vec{j}$ and $\vec{t_0}$ such that $\vec{t} \mapsto\Phi_{\vec{j}}(\vec{t})x$ is a submersion at $\vec{t_0}$ and $\Phi_{\vec{j}}(\vec{t_0})=y$. 
\item[(b)] \cite[Thm.~3.3]{PSComplexity} (Chow theorem) If for each $x\in M$ the vector fields $X_1, \dots, X_p$ together with their brackets generate $T_x M$ then
$\cX\in \fA.$
\item[(c)]\label{item:genericity_generate}
For each $p\geq 2,$ $\fA$ is open and dense  
in the space of $C^{\infty}$ vector fields.
\end{enumerate}
\end{theorem}

 Part (c) is proven in \cite{Lobry70} for dissipative vector fields, however, the argument also works
in the divergence free case. We sketch the argument here since similar reasoning will be used to  
prove Theorem \ref{ThGenVF-EoA}. It sufficient to consider the case $p=2$.

 We encode the failure of accessibility by the union of large codimension submanifolds $W$ of a jet bundle 
  and then use Thom Jet Transversality Theorem \ref{ThThom}. 
 If the codimension is sufficiently large, transversality then implies that a residual subset of maps have their jet disjoint from $W$. $W$ will be unions of 
 submanifolds of the bundle of jets of sections $M\to TM\oplus TM$.

Roughly the proof will show the following: if we have a pair of vector fields, then generically at every point $z\in M$, either $X_1(z)$ or $X_2(z)$ does not vanish. Moreover, if one, say $X_1$ does not vanish at $z$, then we prove that linear relations among the Lie derivatives $\mc{L}^i_{X_1}X_2$ are a positive codimension in the space of jets and hence for sufficiently large $k$ the lie derivatives $\mc{L}_{X_1}X_2,\ldots, \mc{L}^k_{X_1}X_2$ must span $T_zM$.

First we define a subset $W_i^1$. This is the subset of $j^{i+d}(M,TM\oplus TM)$ of jets of vector fields $(X_1,X_2)$ such that $(z, j^{i+d} X_1, j^{i+d} X_2)\in W^i_1$
if 
$X_1(z)\neq 0$ and $\{\mc{L}_{X_1}^iX_2(z),\ldots,\mc{L}_{X_1}^{i+d}X_2(z)\}$ are not linearly independent. 
 We claim that $W_i^1$ is a finite union of submanifolds of positive codimension.

The claim is easiest to see in coordinates (the codimension is coordinate independent). 
For a given choice of $X_1\neq 0$ at $z$, we can pick linearizing coordinates so that $X_1=e_1$. Then the condition that $\{\mc{L}^i_{X_1}X_2(z),\ldots,\mc{L}^{i+d}_{X_1}X_2\}$ span $T_zM$ 
 is equivalent in coordinates to the condition that the columns of the matrix of vectors $[\partial^i_{e_1}X_2,\ldots,\partial ^{i+d}_{e_1} X_2]$ are  not linearly independent. The failure of linear independence is equivalent to the rank of this matrix  being $q$ for some $q<d$. 
 The condition that the matrix has rank exactly $q$ is the condition that all $q \times q$ minors containing a non-vanishing minor of order $(q-1)\times (q-1)$ have rank $0$ \cite[~2.2.1]{cieliebak2024introduction}. 
 Because each $q\!\times\!q$ minor contains a non-vanishing $(q-1)\!\times\!(q-1)$ subminor, each of these vanishing locuses gives us a submanifold $Q(X_1)$ of $j^{d+i}_z(M,TM)$ of codimension at least $1$ depending, in these coordinates, only on the partial derivatives of $X_2$ of order between $i$ and $i+d$. (Note moreover, that this argument applies even though we are restricting to the jets of conservative vector fields because jets we are considering only involve the derivative in one direction.)  
 Further, note that $Q(X_1)$ varies smoothly with $X_1(z)$. 
 Define $W_i^1(X_1)$ to equal the union of the submanifolds $Q$ corresponding to the various minors just described.

 Having established that
$W_i^1$ is the union of submanifolds of $j^{d+i}_z(M,TM\oplus TM)$ of codimension at least $1$,
consider  $W^1=W_{d}^1\cap W_{d+2d}^1\cap \cdots \cap W_{d+2d^2}^1$. As these conditions are independent of each other, the codimensions add. 
(In the coordinates described above, $W_1$ is literally a product of $W_{d+2kd}^1$ for $k=0,\dots, d$.) Thus,
$W^1$ is a finite union of submanifolds of codimension $d+1$ of $j^{d+2d^2}_z(M,TM\oplus TM)$. Similarly, we may define a subset $W_2$ with the reversed condition, that for certain ranges of $i$ the $\mc{L}^i_{X_2}X_1$ fail to form a basis of $T_zM$ when $X_2$ is non-vanishing.

Then by the Thom jet transversality theorem, it follows that for a generic pair of conservative vector fields $(X_1,X_2)$ over $M$
is transverse to $W^1\cup W^2$. As $W^1\cup W^2$ is codimension $d+1$ subset of the jet bundle, $X_1$ and $X_2$ are thus generically disjoint from $W^1\cup W^2$. This means that generically, if $z$ is a point where one of the vectors fields, say $X_1$, does not vanish, then $\{\mc{L}^i_{X_1}X_2\}_{1\le d(1+2d)}$ will span $T_zM$. From jet transversality, it is also the case that each of $X_1$ and $X_2$ has finitely many zeros and they occur at distinct points (a generic section of $TM\oplus TM$ avoids the zero section and is transverse to $\{0\}\oplus TM\cup TM\oplus \{0\}$). Thus for a generic pair of vector fields, at every point $z$ one of them does not vanish, and the brackets of that field with the other generate $T_zM$. 
Hence the accessibility is generic for pairs of volume preserving vector fields. 

\begin{lemma}
(cf.~\cite{BH12}, \cite[Thm.~4.4]{BBMZ})
\label{LmAccUE}
If $\cX\in \fA$ then the volume is the unique stationary measure for $\mu_{\cX}$.
\end{lemma}

\begin{proof}
    Given $x$ and $y\in M$, let $\vec j, \vec t_0$ be as in Theorem \ref{ThControl}. Assume first that $T$ is large enough so that 
    for each $x$ and $y$ the absolute value of each component of $\vec t_0$ is less than $T$. Split $t=(t', t'')$ 
    so that $t''$ is $d$ dimensional and $\DS \det\left(\frac{\partial \Phi_{\vec j}}{\partial t''}\right)(\vec t_0)\neq 0.$
    Then for $\vec t$ close to $\vec t_0$ this determinant is also non zero. Integrating over $t'$ as above
    we see that for all $x$ and $y$ the distribution of $f_\omega^r x$ has density which is positive in a neighborhood
    of $y$. It follows that the Markov chain $x\mapsto f_\omega^r x$ satisfies the Doeblin condition \cite[p.~402]{meyn2009markov} and so its stationary
    measure is unique. This completes the proof in the case $T$ is sufficiently large. In the general case
    consider $m$ such that all components of $\vec t_0$ have absolute value less than $Tm$ and consider the event that
    for $0\leq k<r$ vector field $X_{j_{k+1}}$ is applied during the steps $kr+1, \dots ,(k+1) r$
    where $j_1, \dots,  j_r$ are components from Theorem \ref{ThControl}(a).
 \end{proof}

As a shorthand below, we will say \emph{geometric structures} to refer to measurable families of bundles or 
Riemannian metric as in the criterion in Proposition \ref{prop:EoA_characterization}(b). 

\begin{lemma}
\label{LmSmoothStructure}
Suppose $\cX\in \fA$ and $\mc{E}$ is a bundle over $M$. If there is a measurable geometric structure defined on $\mc{E}$ that is $\mu$ almost surely invariant,
then there is a smooth structure which is $\mu$ almost surely invariant.  
\end{lemma}

\begin{proof}
Given $x$ and $y$ let $\vec j, t', t''$ be as in the proof of Lemma \ref{LmAccUE}. Let $q_x$ be the invariant measurable structure given by
our assumption. By Fubini Theorem for almost every $x$ there is $\bar t'$ arbitrary close to $t'_0$ such that for
almost all $t''$, 
$\DS \Phi_{\vec j}(\bar t', t'') q_x=q_{\Phi_{\vec j} (\bar t' t'') x}$. Note that the left hand side is a smooth function of
$\tilde y=\Phi_{\vec j}(\bar t', t'') x.$ It follows that $q_{\tilde y}$ coincides almost surely with a smooth version in
a small neighborhood of $y$. By compactness it follows that there exists a continuous structure $\bar q_y$ which coincides with 
$q_y$ almost everywhere. By continuity $\bar q$ is $\mu$ invariant.
\end{proof}

\begin{lemma}
\label{LmSkewDensity} Suppose that $\cX\in\fA$, $\cE\to M$ is a vector bundle, and $F$ is a random map of $\cE$ covering $\mu$.
\begin{enumerate}[leftmargin=*]
    \item 
If there exists $x\in M$  such that for all non zero $v\in \cE_x$ the law of the image of $(x,v)$ has 
 an absolutely continuous component on $\mc{E}$, then there are no invariant 
 smooth geometric structures for $\cX.$

\item (cf.~\cite{elliott2023uniformly}) Suppose that for some $x\in M$, the law of the pair $(fx, F_x)$ has density on
$M\times \SL_d(\mathbb{R})$ then there are no invariant smooth
geometric structures.

\item Let $\fU$ denote the Lie algebra  generated by $X_1, \ldots , X_p.$
Suppose that there exists a chart, a point $x\in M$, and vector fields 
$Z_1, Z_2, \dots ,Z_q\in \fU$ with $q=d+d^2-1$ such that 
the vectors 
$$ \left\{\left(Z_j(x), D_xZ_j \right)\right\}_{j=1}^q$$ 
generate $\R^d\times \mathfrak{sl}_d$. Then there are no $\cX$ invariant geometric structures.
\end{enumerate}
\end{lemma}

\begin{remark}
This lemma can be applied to check that certain random systems are (co)expanding on average. As we saw above, if
$\cX\in \fA$ then the corresponding Markov process on $M$ satisfies the Doeblin condition which directly gives a
spectral gap in $L^2.$ However, the results of \S \ref{SSStability} show that the spectral gap also persists for small (in a weak topology) perturbations
of $\mu$ which is a new result, 
 cf.~Remark~\ref{RmNotOpen}.    
\end{remark}

\begin{proof}[Proof of Lemma \ref{LmSkewDensity}.]
(a) Suppose there is an invariant (finite) subbundle family $\cF\subset \cE.$ Then taking $v\in \cF_x$ we see that
$F_x(v)\in \cF_{fx}$ has finite support and so its law cannot have a component with a 
density. Likewise if $q_x$ is an invariant metric 
then $q_x(v, v)=q_{fx}(F_x v, F_x v)$, again precluding $(x, F_x v)$ from having an absolutely continuous law.

(b) follows from (a) since if $A$ has a density on $\SL_d(\reals)$ and $v$ is a non zero vector then 
the law of $Av$ is absolutely continuous as well.

(c) Let $q$ be an invariant geometric structure.
Note that the space of vectorfields preserving $q$ forms a Lie algebra.
Now given $\cZ=\{Z_1,\ldots Z_q\}$ as in the assumption of the lemma, and $T>0$,
then the random dynamical systems defined by random motion along the fields $\cZ$ also preserve $q.$ However, the distribution of $(f(x), Df(x))$ has an absolutely continuous component in $\mathbb{R}^d\times \SL_d(\mathbb{R})$
and so by already proven part (b), $\cZ$ cannot preserve $q$.
\end{proof}

We are now ready to prove Theorem \ref{ThGenVF-EoA}. 

\begin{proof}
    By Theorem \ref{Th-PuSh}(c), accessibility is generic, hence by Lemma \ref{LmAccUE} we know that generically volume is the unique invariant measure, so by Lemma \ref{LmSmoothStructure} it suffices to verify that generically the criterion
of Lemma \ref{LmSkewDensity}(c) is satisfied. As before, we can check this using jet transversality
 which in fact gives a stronger statement that the condition of the theorem generically holds for all $x\in M$. 
We will not give a detailed argument, but explain why the argument of Theorem \ref{Th-PuSh}(c) extends to this case. If we have two vector fields $X_1,X_2$, then if $X_1$ is non-vanishing, then in coordinates we may write $X_1=e_1$. What we want to show is then that the vector fields $Y_i\coloneqq ({\partial^i_{e_1} X_2,\partial^i_{e_1} DX_2})$ span  $\R^d\times \mf{sl}_d$. A slight complication now arises because the pairs $(Y_i,DY_i)$ cannot be chosen arbitrarily as the second term is the derivative of the first. However, note that we can instead restrict to even numbered indices $Y_2,Y_4,\ldots$ and still be able to choose the entries of these matrices freely. Analogously to before, we can define a subset of the jet bundle $W_i^1$ according to the condition that $\{(Y_{i+2k},DY_{i+2k})\}_{1\le k\le d+d^2-1}$ fail to span $\R^d\times \mf{sl}_d$. As before, this is a positive codimension condition that is given by a union of submanifolds of the jet bundle. 
Letting $D\!=\!4d\!+\!(d\!+\!d^2\!-\!1)$
we see that the subset $W^1\!=\!W^1_d\cap W^1_{d+D}\cap \cdots \cap   W^1_{d+(d+1)D}$  has codimension $d+1$ in the jet bundle $j^{d+(d+2)D}(M,TM\oplus TM)$. By jet transversality, we can now similarly conclude.
\end{proof}

 \begin{remark}
In Theorem \ref{ThGenVF-EoA} the amount of time that each vector field is applied for is uniformly bounded by $T$. There are several models considered in the literature where the times are unbounded. For example, in piecewise deterministic Markov chains
\cite{BH12, BBMZ, Davis84} the switching time has an exponential distribution. In the opposite direction 
one can make the switching rate go to zero obtaining stochastic PDEs studied in 
\cite{Bax86, Bax89, BS88, Car85}. In all those models expansion on average is also generic, however, 
we cannot immediately apply Theorem \ref{thm:essential_spectral_gap} since the corresponding measures
are not concentrated on a compact set. It is likely that this $C^k$ norms could be controlled 
using appropriate growth estimates for the solution of the linearized equation, but we do not pursue this topic
here in order to simplify the presentation.
 \end{remark}

 Theorem \ref{ThGenVF-EoA} allows us to construct coexpanding on average systems in a small neighborhood of the identity. Similar ideas could be used to construct coexpanding on average systems 
near an arbitrary diffeomorphism. Here we give one example.

\begin{example}
   Let $f$ be a volume preserving diffeomorphism of a compact manifold $M$ and $X$ be a divergence free
  vector field on $M$. Fix $T>0$ and let $\mu_{f, X, T}$ be the law of $\Phi_X(t) \circ f$ where $t$ is uniformly distributed on $[-T, T]$. 
\end{example}

\begin{theorem}
\label{ThXF}
Suppose that $f$ is a diffeomorphism such that for each $\ell$,  $f$ has only finitely many periodic points of period $\ell$. Then for an open and
dense set of vector fields $X$, $\mu_{f, X, T}$ is coexpanding on average and, moreover, this measure is coclean in the sense of Definition \ref{DefClean}.
\end{theorem}

\begin{proof} \noindent\textbf{Step 0.} First we introduce some notation. Let $P_k$ be the set of points in $M$ of period less or equal to $k$ for $f$. By assumption this is a finite set. Let $x_n$ be the distribution of the point $x_0$ after $n$ iterates of the random dynamics driven by $\mu_{f,X,T}$. The plan of the argument is to check the criteria in Definition \ref{DefClean} of cocleanness, as it then follows that the resulting dynamics is coexpanding on average by Corollary \ref{rem:clean}. 

\noindent\textbf{Step 1.} 
Write $X^t$ for the time $t$ flow of the vector field $X$. Note that we can rewrite the dynamics of $\prod_{i=1}^k X^{t_i}f$ by pushing the vector fields through $f$. Let a vector $\vec{t}=(t_1,\ldots,t_k)\in \R^k$ give the durations of the random flows. We then define
$\displaystyle
\Phi^k(\vec{t}\,)\coloneqq X^{t_k}\hat{X}^{t_{k-1}}_{k-1}\cdots \hat{X}^{t_{1}}_1f^k(z) 
$
where 
$\displaystyle
\hat{X}^t_i=((Df^{k-i})_* X)^t.
$

Let us first show
that the image of a point $f^{-k}(z)$ has absolutely continuous component containing $z$ in its support. To do this, it suffices to show that for $k$ sufficiently large it is generic that $D\Phi^k(\vec{0})$ has rank $d=\dim M$.

For $\Phi$ to have rank $d$, it suffices that $X,\hat{X}_{k-1},\ldots,\hat{X}_1$ span $T_zM$. Note that if $z$ is a point with $f^{-k}(z),\ldots,z$ distinct, then having $X,\hat{X}_{k-1},\ldots,\hat{X}_1$ span $T_zM$ is a constraint on the $1$-jet of $X$ at the points $f^{-(k-1)}(z),\ldots,z$. Note that if $k=d$, then the codimension of failing to span is codimension $1$ in the space of jets. Moreover, when $k=d+j$ the condition that $X,\hat{X}_{d+j-1},\ldots,\hat{X}_1$ fail to span $T_zM$ is codimension $j$ in the space of jets: this is the codimension of the condition that the all $d\times d$ minors of a $(d+j)\times d$ matrix have determinant zero. In particular, for $k=2d+1$, the condition is codimension $d$. Let $W_{2d+1}$ be the space of jets of vector fields $X$ such that $X,\hat{X}_{2d},\ldots,\hat{X}_1$ do not span $T_zM$ for all points $z\in M\setminus P_{2d+2}$. Moreover, similar to the proof of Theorem \ref{Th-PuSh}, $W_{2d+2}$ is the union of finitely many manifolds in the space of $1$-jets of codimension $2d+1$. Thus by the Jet transversality theorem (Theorem \ref{ThThom}), 
 a generic vector field $X$ is transverse to $W_{2d+1}$, and hence is disjoint from $W_{2d+1}$ because $W_{2d+1}$ is codimension $d+1$. In particular, it implies that for all $x_0\notin M\setminus P_{2d+2}$, the law of $x_{2d+1}$ has an absolutely continuous component containing $f^{2d+2}(x_0)$.

To see that generically for every $x_0\in M$,
the law of $x_{2d+2}$ has an absolutely continuous component note that generically $X$ does not vanish on $P_{2d+2}$, hence almost surely $x_1\notin M\setminus P_{2d+2}$, so we can apply the result of the previous paragraph. This gives the needed conclusion for the distribution of $x_{2d+2}$.

\noindent\textbf{Step 2.} Next, we check that there exists $k_d\in \N$ such that for all $(x,v)$ with $v\in T^1_xM$, the unit tangent bundle of $M$, the distribution of $(x_{k_d},v_{k_d})$ has an absolute continuous component as long as $d'$ is sufficiently large. We omit a detailed argument, as it is similar to the proof of Lemma \ref{LmSkewDensity}, and is an elaboration of the argument in the previous step.  

\noindent\textbf{Step 3.} Next we show that volume is ergodic. 

We claim that if $\Omega$ is an  invariant set for $\mu_{X,f,T}$-almost every map then it is also invariant by both $f$ and the flow $X^t$.
Indeed, for almost every $(t_1, t_2)\in [-T, T]^2$ we have $X^{t_1} f(\Omega)=X^{t_2} f(\Omega).$ It follows that for almost every 
$t\in [-2T, 2T]$ we have $X^t f(\Omega)=f(\Omega).$ Since the set of $t$s such that this equality holds is closed 
by Proposition \ref{PrNoInvSetsDiff},
$f(\Omega)$ is preserved by the flow of $X$. Hence for almost every $t$,\;
$\Omega=X^t f(\Omega)=f(\Omega)$ so $f$ preserves $\Omega$ as well.

We now show that the random dynamics generated by  $\Phi^{2d+1}(\vec t)$ from Step 1 is ergodic. From this ergodicity of $\mu_{f,X,T}$ follows easily:  if $\Omega$ is the invariant set as above then $\vol(\Omega)\in \{0, 1\}$ and the ergodicity
follows from Proposition \ref{PrRandomET}.

So let $\hmu$ be the measure on $\Diff_{\vol}^\infty(M)$ defined by $\Phi^{2d+1}(\vec t)$ and consider the Markov process 
$\{y_n\}$ on $M$ defined by  $y_n=g_n y_{n-1}$ where $\{g_n\}$ are IID diffeomorphisms distributed according to $\hmu$.
We will show that this process is exponentially mixing in the sense that for each $y', y''\in M$ the measures
$\hmu^n*\delta_{y'}$ and $\hmu^n*\delta_{y''}$ are exponentially close with respect to the variational distance.
To this end it suffices to show that there exists $n_0$ and a ball $B\subset M$ such that for each $n\geq n_0$ there is
a constant $\rho_n$ such that 
for each initial point $y_0$ the distribution of $y_n$ has an absolutely continuous component with density bounded
from below by $\rho_n$. We first show this when initial state is bounded away from $P_{2d+1}$. 
Note that the proof of Step 1 shows that generically for each $y_0\not \in P_{2d+1}$ the distribution of $y_1$
has a continuous component with density positive in a ball centered at $y_0$ with radius $r(y_0)$.  
Let $G_\eta$ be the set of points whose distance from $P_{2d+1}$ is at least $\eta$. By compactness there exists $\bar r$ such 
that $r(y_0)\geq \bar r$ for $y_0\in G_\eta$ and moreover the density on the corresponding components is
at least $\bar \rho.$
Decreasing $\bar r$ if necessary we can find a small ball $B(\bar y, \bar r)$
which is completely contained in $G_\eta$. 
Since $M$ is connected for small $\eta$ there exists $n_1$ such that for each $y_0\in G_\eta$ there exists
a sequence $y_0, y_1, \dots , y_{n_1}=\bar y$ such that the distance between the consecutive points is less than $\bar r/3$.
This proves the claim for $y_0\in G_\eta$ and $n\geq n_1$ with the lower bound on the density equal to 
$\left(\bar\rho \min_{y\in M} \vol(B(y, \bar r/3))\right)^{n_1}.$
Next, if $\eta$ is sufficiently small then there exists $q>0$
such that for each $y_0\not\in G_\eta$ the probability that $y_1\in G_\eta$ is at least $q$ proving the result for all
$y_0\in M$ with $n_0=n_1+1$. 
\vskip2mm

\noindent\textbf{Step 4.} We now conclude that $\mu_{f,X,T}$ is coclean. Indeed, we showed in Step 1 that the distribution of $x_{2d+d}$ has an absolutely continuous component. In Step 3 we showed that volume is ergodic, and in Step 2, because $(x,v)$ has an absolutely continuous component there cannot be a volume measurable invariant subbundle of $T^*M$. Hence the measure $\mu_{f,X,T}$ is coclean and thus coexpanding on average by Corollary \ref{rem:clean}.
\end{proof}

Using similar ideas, we can show that the Pierrehumbert model studied in \cite{BCZG} is coexpanding on average. The model originates in the paper \cite{pierrehumbert1994tracer}. The Pierrehumbert model is a random  
composition of vertical and horizontal sinusoidal shears, where the shears each have have independent, uniformly random phase shifts. Formally, this model is described as follows. 

\begin{example}
\label{ExPierrehumbert}
    Let $\mathbb{T}^2=[0,2\pi)^2$ be the torus
and $\tau$ be a positive parameter.
Then we define two measures $\mu_H$ and $\mu_V$ on $\Diff^{\infty}_{\vol}(\mathbb{T}^2)$. The measure $\mu_H$ is given by the pushforward of normalized Lebesgue measure on $[0,2\pi)$ by the map 
\[
t\mapsto (x,y)\mapsto (x+\tau\sin(y+t),y),
\]
and $\mu_V$ is the pushforward of the normalized Lebesgue measure on $[0,2\pi)$ by 
\[
t\mapsto (x,y)\mapsto (x,y+\tau\sin(x+t)). 
\]
Then the Pierrehumbert model is the random dynamics of the measure $\mu=\mu_V*\mu_H$. 
\end{example}

\begin{proposition}
The Pierrehumbert model is coexpanding on average.
\end{proposition}
\begin{proof}
Due to Corollary \ref{prop:low_dim_equivalence} it suffices to check that $\mu$ is expanding on average. 
We verify this by checking the criterion in Proposition \ref{prop:EoA_characterization} for conservative maps: 
we show that there is no measurable a.s.~invariant Riemannian metric or family of vector bundles. An easy computation shows that for the  Pierrehumbert system the distribution of $f(x)$ is absolutely
continuous for each $x$ and, moreover, the system is accessible. Therefore by the argument of Lemma \ref{LmSmoothStructure} a measurable structure can be promoted to a smooth one, so it suffices to show that
there are no smooth invariant geometric structures of either of the two types mentioned above.

Suppose that $V$ is a smooth family of  differentiable line bundles that is almost surely invariant under $\mu$. 
By Remark \ref{rem:two_subbundles} it follows that over any point $V$ may contain at most two lines. Since the action on Grasmannians is one-to-one 
the number of lines does not depend on the point.

Note that
the image of $q\coloneqq (0, 0)$ under $\mu_V$ is equal to 
$([-\tau,\tau]\,\, \mathrm{mod}\,\, 2\pi \mathbb{Z})\times \{0\}$ and many images have multiplicity at least two $2$. 
Moreover the differentials at these images are different. Namely, if $ p(z)\coloneqq (z,0)\in  [-\tau, \tau] \times \{0\}$, then there exist two shears $f_1$ and $f_2$ such that 
$f_1q=f_2q=p(z)$ and
\[
Df_1(q)=
\begin{bmatrix}
1 & d(z) \\ 
0 & 1
\end{bmatrix}
\quad 
Df_2(q)=
\begin{bmatrix}
1 & -d(z) \\
0 & 1
\end{bmatrix},
\]
where $d(z)=\sqrt{1-z^2}.$  
 Below we introduce the slope coordinate on the unit tangent bundle defined by $\zeta=x/y.$ In these coordinates,
 the matrix $\DS \begin{bmatrix}
1 & d \\ 0 & 1 \end{bmatrix} $ acts on the projective space by $\zeta\mapsto \zeta+d$.

 Suppose that the family $V(\cdot)$ consists of two lines. Call their slopes $L_1$ and $L_2$. 
 We claim that it then follows that $V$ contains more than two line contradicting Remark \ref{rem:two_subbundles}. Indeed consider $z$ 
 where $d(z)$ is defined and not equal to $0$. 
There are two cases:

\noindent
(i) if $\zeta_1\!:=\!L_1(q)\!\neq\! \infty$ and $L_2(q)\!\!=\!\!\infty$, then $Df_j(q) L_2=\infty$ while $Df_1(q) \zeta_1\!\!\neq\!\! Df_2(q) \zeta_1$;

\noindent (ii)
 If $\zeta_1=L_1(q)<\zeta_2=L_2(q)$ are both finite then $\zeta_1-d(z)<\zeta_1+d(z)<\zeta_2+d(z)$. 

So, in either case we get at least three lines.

If $V$ consists of a single line, then similarly to the case (i) above $V(q)$ should be vertical.
But the same reasoning applied to $\mu_H$ shows that $V(q)$ must be horizontal giving a contradiction.

The case of an invariant measurable Riemannian metric is similar. If we had such a metric $g$
we can represent $g_q$ by a quadratic form corresponding to a matrix
$\DS
\begin{bmatrix}
a & b\\
b & c
\end{bmatrix}.
$
Then the pushforwards of this metric by $f_1$ and $f_2$ to $(z,0)$ correspond to the the matrix 
\[
\begin{bmatrix}
1 & 0\\
\pm d(z) & 1
\end{bmatrix}
\begin{bmatrix}
a & b\\
b & c
\end{bmatrix}
\begin{bmatrix}
1 & \pm d(z) \\
0 & 1
\end{bmatrix}=
\begin{bmatrix}
a & \pm ad(z)+b\\
\pm ad(z)+b & ad^2(z)\pm 2bd(z)+c
\end{bmatrix}.
\]
As the images of $g_q$ by $Df_1(q)$ and $Df_2(q)$ should coincide we must have
$a=b=0.$ A similar argument using the horizontal shears gives $c=0.$
\end{proof}

 We close this subsection with an additional example, showing that a notoriously difficult to study system, the Chirikov-Taylor standard map, becomes expanding on average after perturbation. 

\begin{example}
The following random system on $\mathbb{T}^2$ is considered in \cite{blumenthal2017lyapunov}
\begin{equation}
\label{PertHenon}
f(x,y)=(L\psi(x)-y+\omega, x) 
\end{equation}
where $\psi:\mathbb{T}\to \mathbb{R} $ is a function such that all critical points of both $\psi$ and $\psi'$ are non-degenerate (and, hence there are finitely many such points),
and $\omega$ is uniformly distributed on $[-\eps, \eps].$ 
\end{example}
\begin{proposition}
\label{PrStandard}
\cite{blumenthal2018lyapunov}
    Given $\delta>0$,  and $\psi$ as above there exists $L_1$ such that if $L\geq L_1$ and $\eps>L^{\delta-1}$ then the random system \eqref{PertHenon} is coexpanding on average.
\end{proposition}

\begin{proof}
    By Proposition \ref{prop:coa_d-1_planes} it suffices to show that the above system is expanding on average. To this end we note that
    \cite[Prop.~9]{blumenthal2018lyapunov} shows that integral \eqref{eqn:1Planes} is bounded from below by a quantity of order $\ln L$
     for every stationary measure on the projective extension of \eqref{PertHenon}.
\end{proof}

We note that the results of \cite{blumenthal2017lyapunov, blumenthal2018lyapunov} are much stronger than Proposition~\ref{PrStandard}. In particular they get some information about the size of Lyapunov exponents
and they can handle the dissipative systems where the second component in \eqref{PertHenon} equals $bx$ for $b\neq 1$. 
The results of our paper show in particular that mixing obtained in \cite{blumenthal2018lyapunov}
persists for small weak* perturbation of \eqref{PertHenon}. 
 In particular, it persists for discrete approximations (of a sufficiently large cardinality). In this respect we would like to mention that
\cite{chung2020stationary} constructs explicit discrete perturbations of the standard map 
which are (co)expanding on average.

\subsection{Homogeneous systems and their perturbations}
\label{ScHomogeneous}

In this section, we explain that many algebraic systems as well as their perturbations are coexpanding on average.
The expanding on average property has been known for random matrix products for a long time. For example, if $\mu$ is a compactly supported measure on $\SL(d,\R)$ that is strongly irreducible and contracting, the random matrix products arising from $\mu$ are expanding on average \cite[Cor.~III.3.4]{BougerolLacroix}  (Recall that a linear action is called strongly 
irreducible if it does not preserve a family of linear subspaces, and it is called contracting if it does
not preserve a positive definite quadratic form). It was observed in \cite{goldsheid1989lyapunov} that 
the invariant structures described above are defined by polynomial equations and so the irreducibility and contraction
properties hold if the support of $\mu$ generates a Zariski dense subgroup of $\SL(d,\R)$.

\begin{example}
\label{ExTorus}
Consider the following diffeomorphisms of $\mathbb{T}^d$: $f_j(x)=A_j x+b_j$ where $A_j$ are elements of
$\SL(d, \Z)$ and $b_j$ are vectors in $\mathbb{T}^d.$
\end{example}

\begin{proposition}
\label{PrTorusAffine}
    If the group generated by $(A_1, \dots , A_m)$ is Zariski dense then the above tuple is coexpanding on average and mixing.
\end{proposition}

\begin{proof}
The corresponding action on $T^* \mathbb{T^d}$ is given by $((A_1^T)^{-1}, \dots, (A_m^T)^{-1})$
which also generate a Zariski dense subgroup. So by the foregoing discussion this action $(f_1,\dots ,f_m)$
is coexpanding on average. 
To show that the action is mixing it suffices to show for each $k_1, k_2\in \Z^d$
$$ \mathbb{E} \left(\int \exp(2\pi i \langle k_1, f_\omega^n x\rangle ) \exp(2\pi i \langle k_2,  x\rangle ) dx \right)
\to 0 $$
   as $n\to\infty.$ However, the above expression equals to
   $$\EXP \left(\int \exp(2\pi i(\langle S_n(\omega) k_1+b_n(\omega)+k_2, x \rangle) dx \right) $$
   where $S_n({\omega})=A_{\omega_n}^*\dots A_{\omega_1}^*$ is the linear part and $b_n(\omega)$ is the corresponding translational part.
   Since the action of $(A_1^*, \dots , A_m^*)$ is expanding on average $\|S_n k_1+k_2\|$
   tends to infinity almost surely, and hence the probability that $S_n k_1+k_2=0$ goes to $0$ as $n\to\infty$.
\end{proof}

\begin{example}\label{ex:homogeneous}
Let $G$ be a real algebraic semisimple group without compact factors, and consider  the action of $G$ by left translation on
$M=G/\Gamma$ where $\Gamma$ is a cocompact lattice. Let $\mu$ be a measure
supported on a compact subset of $G$ and consider random translations on $M$
$x\mapsto gx$, where $g\in G$ is distributed according to $\mu.$
\end{example}

\begin{proposition}
\label{PrHomogeneous}
Let $H$ denote the Zariski closure of the group generated by $\supp(\mu).$
If $H$ is semisimple with no center and no compact factors, then $\mu$
is expanding and coexpanding on average and mixing.
\end{proposition}

\begin{proof}
The proof is similar to the proof of Proposition \ref{PrTorusAffine} but we use the adjoint representation
of $G$ instead of the natural action of $\SL_d(\R)$ on $\R^d.$

The expansion and coexpansion on average follow from \cite[Remark on p. 3]{eskin2018random}.
To see that the volume is mixing we need to show that for each pair of zero mean  $L^2$ functions $\phi$ and $\psi$ on $M$
$$ \EXP\left(\int\phi(x) \psi(S_n x)dx \right)\to 0$$
where $S_n=g_n\dots g_1$ and $\{g_n\}$ are IID distributed according to $\mu.$
From expansion on average it follows that projection of $S_n$ on each simple factor of $G$
tends to infinity, so by the Howe--Moore Theorem \cite[Thm.~2.2.20]{zimmer1984ergodic} the expression in parenthesis tends to $0$
almost surely proving mixing.
\end{proof}

\begin{remark}
    In fact much stronger results are known for Examples \ref{ExTorus} and \ref{ex:homogeneous}.
    In particular, \cite[Thm 1.1]{benoist2011mesures} tells that volume and periodic measures
    are only invariant measures for $\mu$, which is much stronger than mixing.

Also a minor modification of the proofs of Propositions \ref{PrTorusAffine} and \ref{PrHomogeneous} using the
large deviations bounds (see \cite[\S 12.5]{BQ-Book}) shows that the actions of those examples are, in fact, exponentially mixing.

Theorem \ref{thm:essential_spectral_gap} gives a different proof of exponential mixing, which also works for small non linear perturbation
of Examples \ref{ExTorus} and \ref{ex:homogeneous}.
\end{remark}

\begin{example}
Small perturbation of isometries were studied in \cite{dewitt2024simultaneous, dolgopyat2007simultaneous}.
 The following dichotomy is obtained.  

\begin{theorem}\label{thm:coa_isometry_perturbation}
 Suppose that $M$ is an isotropic manifold of dimension at least $2$ and let $(R_1,\ldots,R_m)$ be a tuple topologically generating the connected component of the identity of the isometry group of $M$. Let $(f_1,\ldots,f_m)$ be
 a $C^{\infty}$ small volume preserving perturbation of $(R_1,\ldots,R_m)$.
 Then either the perturbed maps are simultaneously conjugated back to isometries, or the perturbed random system is
 is both expanding on average and coexpanding on average.
 \end{theorem}

This fact is not stated explicitly in these papers, so we will sketch the argument here,
even though it has been known to the experts for some time.

By Proposition \ref{prop:EoA_int_characterization}, in order to check the expansion on average condition, we need to verify that for 
all stationary measures $\nu$ on $\PP(TM)$, that on the projectivization of the tangent bundle of $M$ the following integral is  strictly positive: 
\begin{equation}\label{eqn:integral_on_proj}
\iint \sum_{i=1}^n \ln \|Df_iv\|\,d\nu(v)\,d\mu(f)>a
\end{equation}
 for some $a>0$.

The main argument in \cite{dolgopyat2007simultaneous,dewitt2024simultaneous} is a KAM scheme for producing a conjugacy that simultaneously linearizes the diffeomorphisms $(f_1,\ldots,f_m)$. 
Each step of the KAM scheme is able to proceed as long as there is an ergodic stationary measure $\nu$ for which the integral
\eqref{eqn:integral_on_proj} is 
close to zero in a precise quantitative sense.

This is due to \cite[Prop.~26]{dewitt2024simultaneous}, which gives an expression for the integral of an arbitrary stationary measure $\nu$ on $\PP(TM)$ that is independent of $\nu$ up to 
negligible terms The key feature of the argument in \cite{dewitt2024simultaneous} is that the KAM scheme can proceed as long as 
the first line in equation (18) of \cite{dewitt2024simultaneous} is small compared to the second line. If the KAM procedure can be run indefinitely then the 
$f_j$ are simultaneously conjugated to rotations. If that procedure stops then the main term in 
Prop.~26 comes from the first line of equation (18) and hence it is strictly positive.

Thus if the KAM procedure fails, then \eqref{eqn:integral_on_proj} holds, which shows that $\mu$ is expanding on average. The fact that $\mu$ is also coexpanding 
on average follows from Proposition \ref{prop:coa_d-1_planes} and \cite[Thm. 40]{dewitt2024simultaneous} which shows that the integrals \eqref{eqn:KPlanes} for $k=1$ and $k={d-1}$
are of the same order (note that the term $\Lambda_d$ in \cite[eqn. (93)]{dewitt2024simultaneous} is zero in the volume preserving case).

\begin{remark}
    Note that the same arguments work if we had instead started with the tuple $(R_1^{-1},\ldots,R_m^{-1})$ and its perturbation $(f_1^{-1},\ldots,f_m^{-1})$. Thus if $(R_1,\ldots,R_m)$ is a tuple of isometries of an isotropic manifold as above, and $(f_1,\ldots,f_m)$ is its $C^{\infty}$-small volume preserving perturbation then either $(f_1,\ldots,f_m)$ can be simultaneously conjugated to isometries, or the tuple is expanding on average, coexpanding on average, as well as expanding and coexpanding backwards, too.
\end{remark}

\end{example}

\subsection{Products}\label{subsec:products} In this subsection we show how to construct new examples of coexpanding on average systems from
the existing one. As an application 
we verify that if $\mu$ is expanding on average, then so is the associated $k$-point motion. 

We start by recording several properties of expanding and coexpanding on average systems.

\begin{lemma}
 For a measure $\mu$ on $\Aut(\mc{E})$, 
 the property of being expanding on average is independent of the metric on $\cE$.
\end{lemma}

\begin{proof}
Suppose that $F$ distributed according to a measure $\mu$ is expanding with respect to metric $\|\cdot\|$ and let $\|\cdot\|'$ be a different metric.
The expansion of $\cE$ is equivalent to saying that for each non-zero vector
$\DS \mathbb{E}[\ln \|F^N_\omega v\|]\geq \lambda \ln \|v\|.$ Iterating we see that for each $k\in\mathbb{N}$,
$\DS \mathbb{E}[\ln \|F^{Nk}_\omega v\|]\geq k \lambda \ln \|v\|.$ By compactness there is a constant $C$ such that
for each $v$, $C^{-1}\|v\|\leq \|v\|'\leq C \|v\|.$ It follows that 
$\DS \mathbb{E}[\ln \|F^{Nk}_\omega v\|']\geq k \lambda \ln \|v\|'-2\ln C.$
Taking $k$ large we conclude that $\mu$ is expanding on average with respect to $\|\cdot\|'.$
\end{proof}

\begin{lemma}\label{lem:eoa_product}
Suppose that $M_1,M_2$ are closed manifolds, $M=M_1\times M_2$, and that $\mu$ is probability measure with compact support on $\Diff^1(M)$ that is supported on diffeomorphisms of the form $f(x_1, x_2)=(f_1(x_1), f_2(x_2)).$
Then $\mu$ is (co)expanding on average iff its projections $\mu_j$ to $\Diff(M_j)$ are (co)expanding on average.
\end{lemma}

    Note that the $f_j$  need not be independent. For example, consider {\em $k$-point motion} where
    $M^{(k)}=M\times M\times \dots \times M$ ($k$ times) and $F(x_1, \dots x_k)=(f(x_1), \dots, f(x_k))$.
    Applying Lemma \ref{lem:eoa_product} to this example we obtain: 

\begin{corollary}
\label{CrEoAk}
  The $k$ point dynamics is expanding on average iff the original dynamics is expanding on average.     
\end{corollary}

\begin{proof}[Proof of Lemma \ref{lem:eoa_product}.]
 If $\mu$ is expanding on average then so are $\mu_j$ as follows by considering vectors of the form $(v_1, 0)$ and $(0, v_2)$
 respectively. 
 
 Conversely, suppose that $\mu_j$ are expanding on average. 
  Let $N_j$ be the time realizing the expansion for $\mu_j$ and $\lambda_j$ be the expansion constant. Set $N=N_1 N_2$.
 Consider a metric
\[\|(v_1, v_2)\|'=\max(\|v_1\|, \|v_2\|).\]
 
 Take $v=(v_1, v_2)$ and suppose  that $\|v_1\|\geq\|v_2\|.$ 
 Then
 $$ \mathbb{E}[ \ln \| Df_\omega^N (v_1, v_2)\|']\geq \mathbb{E}[ \ln  \|Df_{1,\omega}^N (v_1)\|]\geq \lambda_1 \|v\|_1=
 \lambda_1 \|(v_1, v_2)\|'.
 $$
The case where  $\|v_1\|\leq\|v_2\|$ is similar. 
\end{proof}
\subsection{Automorphisms of complex surfaces}
 The work of Cantat and Dujardin provides additional examples of coexpanding on average dynamical systems. In \cite[Sec.~9]{cantat2024dynamics}, the authors give examples of random automorphisms of complex surfaces that are expanding on average. In fact, when they are volume preserving, this implies that those automorphisms are coexpanding on average as well. To see this, by Proposition \ref{prop:coa_d-1_planes} it suffices to check that they are expanding on average on $3$-planes as such surfaces have four real dimensions. 
If $\mu$ is an expanding on average measure on $\Aut(X)$ where $X$ is a complex surface, then consider the action of $f\in \Aut(X)$ on a $3$-plane $V$ in $TX$.
Note that we can always choose an orthonormal basis for $V$ of the form $\{v,iv,w\}$. As the map is complex analytic, $\|Dfv\|=\|Dfiv\|$ and $\|Dfw\|=\|Dfiw\|$. Note that $v\wedge iv\wedge w\wedge iw$ is a unit volume form. Hence due to volume preservation 
\[
1=\abs{Df_*(v\wedge iv\wedge w\wedge iw)}=\|Dfv\|\|Dfiv\|\|Dfw\|\|Dfiw\|\sin^2(\angle( Df_*V,Df_*W)). 
\]
Taking the logarithm and expectations over $f$, we obtain 
\begin{equation}
\E{\ln\|Dfv\|}+\E{\ln\|Dfw\|}+\E{\ln \abs{\sin(\angle V,W)}}=0.
\end{equation}
We can now apply this to the expected growth of the volume on a $3$-plane. Note that 
$$
\E{\ln \|Df_*(v\wedge iv\wedge w)\|}=
$$$$
2\E{\ln\|Dfv\|}+\E{\ln\|Dfw\|}+\E{\ln \abs{\sin(\angle V,W)}}=\E{\ln \|Df v\|}.
$$
As observed above, every $3$-plane has a unit volume form of the form $v\wedge iv\wedge w$. Thus if the random dynamics on $1$-vectors is expanding on average, then so is the random dynamics on $3$-vectors. As Cantat and Dujardin note in that paper, this gives a large collection of examples that are far from homogeneous.

\section{Comparing operators using symbols.}
\label{ScSymb-Op}
In this section, we will describe tools for comparing operators by comparing their symbols pointwise. 

We begin with Lemma \ref{lem:square_root} that allows us to essentially take a square root of a symbol. Then we prove a technical lemma that allow us to change the side of an inequality that a compact operator appears on. Finally, we obtain the main result of this section, which compares the norms of operators by comparing their symbols.

\begin{lemma}\label{lem:square_root}
Suppose that for $m\in \R$, that $A$ is an elliptic operator in $\Psi^m(M)$ whose principal symbol is positive for $\|\xi\|\geq 1$. 
Then there exists an elliptic $C\in \Psi^{m/2}(M)$ such that
$\displaystyle
A=C^*C+\mc{K},
$
where $\mc{K}: H^s\to H^{s-m}$ is compact. 
\end{lemma}
\begin{proof}
 Modifying if necessary $A$ 
 by a compact, smoothing operator, we can assume the principal symbol is positive.
 Take $C\!\!=\!\text{Op}(\sqrt{\sigma_A})$
where $\sigma_A$ is the principal symbol of $A$. Then $A\!-\!C C^*\!\in\! \Psi^{m-1}$ and so 
it maps $H^s$ to $H^{s-m+1}.$
\end{proof}

\begin{lemma}\label{lem:move_compact_terms}
Suppose $\mc{B}_1$ and $\mc{B}_2$ are  Hilbert spaces and that $A,B\colon \mc{B}_1\to \mc{B}_2$ are bounded linear operators such that $B$ is Fredholm and there is a compact operator $\mc{K}$ 
such that
\begin{equation}
\label{A-BK}
\|A\phi\|^2\le \|B\phi\|^2+\langle \mc{K}\phi,\phi\rangle. 
\end{equation}
Then for all $\epsilon>0$ there exists a compact operator $\mc{K}_{\epsilon}\colon \mc{B}_1\to \mc{B}_2$ such that 
\[
\|(A+ \mc{K}_\varepsilon)\phi\|\le (1+\epsilon)\|B\phi\|.
\]
\end{lemma}

\begin{proof} 
Since $\mc{K}$ is compact and $B$ is Fredholm  
 there is a finite codimension subspace $\mc{V}$ of $\mc{B}_1$ such that for $\phi\in \mc{V}$,
 \[
\|A\phi\|^2\le \|B\phi\|^2+\langle \mc{K}\phi,\phi\rangle \le (1+\epsilon)^2\|B\phi\|^2.
 \]
 Let $\mc{U}$ be an orthogonal complement to 
$\mc{V}$ with respect to the scalar product $\langle B\phi,B\phi\rangle$. Denoting by $\Pi$ the projection to $\mc{V}$ along $\mc{U}$ we get
$$ \|A\Pi \phi\|^2_0\leq (1+\eps)^2 \|B \Pi \pi \phi\|^2_0\leq (1+\eps)^2\|B\phi\|^2_0 $$
where the first inequality holds since $\Pi\phi\in \mc{V}$ and the second inequality holds by the definition of $\Pi$ using $\mc{U}$.
Since $A-A\Pi$ has finite rank, the result follows.
\end{proof}

\begin{lemma}\label{lem:comparison_of_symbols_and_operators}
Suppose $s\in \R$, $M$ is a closed Riemannian manifold, 
and $A$ and $B$ are  pseudodifferential operators in $\Psi^s(M)$ with associated principal symbols $a(x,\xi)$ and $b(x,\xi)$. Suppose that $B$ is elliptic and that there exist $\lambda$ and $r$ such that for all $x\in M$ and $\abs{\xi}>r$ in $T^*_xM$, $\abs{a(x,\xi)}\le \lambda b(x,\xi)$. Then for all $\epsilon>0$ there exists a compact/smoothing operator $\mc{K}_{\epsilon}\colon H^{-\infty}(M)\to C^{\infty}(M)$ such that for all $\phi\in H^s(M)$,
$$
\|A\phi\|_{0}^2\le (\lambda+\epsilon) \|B\phi\|_{0}^2+\langle \mc{K}_{\epsilon}\phi,\phi\rangle.
$$
\end{lemma}

\begin{proof}
 By definition, we are interested in,
\[
(\lambda+\eps)^2\|B\phi\|^2_0-\|A\phi\|^2_0=\lambda^2\langle A\phi,A\phi\rangle - \langle B\phi,B\phi\rangle. 
\]
Now let $A^*$ and $B^*$ denote the formal adjoints of $A$ and $B$. While not by definition the actual adjoint, these operators are closed and the closure is adjoint to $A$ and $B$ with respect to the (regularized) $L^2$ pairing, see \cite[Sec.~I.8.2]{shubin2001pseudodifferential}, hence
\[
(\lambda+\eps)^2\|B\phi\|^2_0-\|A\phi\|^2_0= \langle ((\lambda+\eps)^2 B^*B-A^*A)\phi,\phi\rangle.
\]
Now by our assumption concerning the symbols, $ (\lambda+\eps)^2 B^*B-A^*A$ is an elliptic operator in $\Psi^{2s}$. Thus by Lemma \ref{lem:square_root}, there exist elliptic $C\in \Psi^{s}$ and compact $\mc{K}$ such that $(\lambda^2 B^*B-A^*A)=C^*C+\mc{K}$. This implies that 
\begin{equation}\label{eqn:est_on_C_K}
(\lambda+\eps)^2\|B\phi\|^2_0-\|A\phi\|^2_0=\|C\phi\|_0^2+\langle \mc{K}\phi,\phi\rangle, 
\end{equation}
which is the needed conclusion.
\end{proof}

\section{Main Estimates}
\label{ScMainProof}

In this section, we prove the essential spectral gap in a series of steps. 
We will concentrate on the spectral gap of $\mu^{-1}$ on $H^s$ for small negative $s$,
the results for $\mu$ follow by duality. 
First, we show how the expanding on average condition relates to a specific estimate on the action of the symbol of the operator $\Delta^{-s}$. Then we use the comparison inequality to compare with the symbol of $\Delta^{-s}$, proving the essential spectral gap. 

\begin{lemma}\label{lem:vectors_expand_lemma}
Suppose that $\mu^{-1}$ is a coexpanding on average measure on $\Diff^1(M)$ with compact support. Then there exists $s_0>0$ and $C$ such that for all $0<s<s_0$, there exists $0<\eta(s)<1$ such that for each $n\in \mathbb{N}$, each $x\in M$ and each $\xi\in T^{1*}_xM$, the unit cotangent bundle
\begin{equation}\label{eqn:upperbound_est}
\int \|{(D_x(f^{-1})^*)^{-1}}(\xi)\|^{-s}\,d\mu^n(\omega)\leq C\eta^n.
\end{equation} 
\end{lemma}

\begin{proof}
We give a proof in the case $N=1$ in the definition of the coexpanding on average property. For other $N$ the proof follows by adjusting the constant $C$.

Define the function $h(s,\xi)\colon (-1,1)\times T^{1*}M\to \R$ by 
\[
(s,\xi)\mapsto \int \|{(D_x(f^{ -1})^*)^{-1}\xi}\|^{-s}\,d\mu(\omega). 
\]
Note that $h(0,\xi)=1$, and that 
$$
\frac{\partial h}{\partial s}({0, \xi})=\frac{\partial }{\partial s} 
\int \|(D(f^{ -1})^*)^{-1} \xi\|^{-s}\,d\mu(\omega)
$$$$
=\int -\ln \|(D(f^{ -1})^*)^{-1} \xi\|\,d\mu(\omega)
<-\lambda<0.
$$
Thus there exists $s_0>0$ such \eqref{eqn:upperbound_est} follows for $s\in (0,s_0]$ for $n=1$ with $C=1$. For larger $n$, the needed conclusion follows by induction.
\end{proof}

\begin{remark}
For a diffeomorphism $f$ there is a natural action on $C^{\infty}_c(M)$ viewed as both functions and distributions. Unless $f$ is volume preserving, the map induced by pulling back a smooth function as a smooth function, and the map pulling back a smooth function as a distribution need not coincide. See e.g.~\cite[Eq.~I.3.13]{treves1980introduction}. This coincidence is used implicitly below.
\end{remark}

The following lemma allows us to combine operators with nonnegative principal symbol. The topology on $S^m(M)$ is the usual Fr\'echet topology on symbols. Below, one can just think of having uniform bounds  in equation \eqref{eqn:S_m_definition} over the entire family.

\begin{lemma}\label{lem:non_negative_sum}
Suppose that $M$ is a Riemannian manifold, $s\in \R$, and that $\{A_i\}_{i\in I}$ is a precompact family of elliptic pseudodifferential operators on $M$ in symbol class $S^{s}(M)$ with non-negative principal symbol, indexed by a probability space $(I,d\mu) $. Then for all $\epsilon>0$, there exists an operator $ B\in\Psi^{s}(M)$ with 
non-negative principal symbol and a (compact) smoothing operator $\mc{K}\in \Psi^{-\infty}$ such that 
for any $\phi\in H^{s}(M)$,
\[
\int \|A_i\phi\|_0^2\,d\mu\le \|B\phi\|_0^2+\langle \mc{K}\phi,\phi\rangle.
\]
and
\begin{equation}\label{eqn:definition_of_B_symbol}
\abs{\sigma_B}^2\le (1+\epsilon)\int \abs{\sigma_{A_i}(\xi)}^2\,{d\mu}.
\end{equation}
\end{lemma}
\begin{proof}
As before, for each $A$ in the support of $\mu$, there is its formal adjoint $A^*$. Then we may write
\[
\int \|A_i\phi\|_0^2\,d\mu= \int \langle A_i \phi,A_i \phi\rangle\,d\mu=\langle \left(\int A^*A\,d\mu\right) \phi,\phi\rangle. 
\]
Define $\hat{B}$ by taking 
\[
\hat{B}=(1+\epsilon)\text{Op}\left(\int \sigma_{A^*}\sigma_A\,d\mu\right).
\]
Then as in the proof of Lemma \ref{lem:comparison_of_symbols_and_operators} because $\sigma_B$ is greater than $(1+\epsilon)$ times the principal symbol of $\int A^*A\,d\mu$, there exists an operator $\mc{K}$ in $\Psi^{-\infty}(M)$, such that 
\[
\int \|A_i\phi\|_0^2\,d\mu\le \langle \hat{B}\phi,\phi\rangle +\langle \mc{K}\phi,\phi\rangle. 
\]
We can then apply Lemma \ref{lem:square_root} to $\hat{B}$ to find $B$ satisfying 
$\hat{B}=B^*B+\widetilde{\mc{K}}$ 
whose symbol satisfies \eqref{eqn:definition_of_B_symbol}. 
\end{proof}

\begin{lemma}\label{lem:action_on_symbol}
Suppose that $\mu^{-1}$ is a coexpanding on average measure on $\Diff^{\infty}_{\vol}(M)$ with compact support. Then for all $0<\lambda<1$,
there exists $n\in \mathbb{N}$ and $r$ such that if we write $\Delta$ for the usual Laplacian, and write  $\sigma_{\Delta^{-s}}(x,\xi)$ for the principal symbol of $\Delta^{-s}$, then for all $x\in M$ and $\abs{\xi}>r$ in $T^*_xM$,
\begin{equation}\label{eqn:comparison}
\int \abs{\sigma_{(\Delta^{-s})^f}(x,\xi)}^2\,d\mu^n(f)\le \lambda \abs{\sigma_{\Delta^{-s}}(x,\xi)}^2. 
\end{equation}
\end{lemma}
\begin{proof}
 Recall from \S \ref{SSPullback} the
change of variables formula 
saying that if $A$ is a pseudodifferential operator 
with principal symbol $a(x,\xi)\colon T^*M\to \R$, and $f\in \Diff^\infty(M)$, 
then the pushforward $A^f$ has principal symbol $b(x,\xi)=a(f^{-1}(x), (D_x(f^{-1})^*)^{-1}\xi)$. 

Let $b_n$ denote the left hand quantity in equation \eqref{eqn:comparison}. 
Choose $2s$ and $n$ such that \eqref{eqn:upperbound_est} in Lemma \ref{lem:vectors_expand_lemma} holds for $C\eta^n<\lambda$. Then, for a unit covector $\xi\in T^*M$,  by the formula for the symbol of the pushforward, \eqref{eqn:pushforward_symbol},
\begin{align*}
b_n(x,\xi)&= \int\abs{\sigma_{\Delta^{-s}}(f^{-1}(x),(D_x(f^{-1})^*)^{-1}\xi)}^2\,d\mu^n (f)\\
&= \int \|(D_x(f^{-1})^*)^{-1} \xi\|^{-2s}\,d\mu^n(f) \\
& \le C\eta^n(s)(\sigma_{\Delta^{-s}}(x,\xi))^2\le \lambda (\sigma_{\Delta^{-s}}(x,\xi))^2. 
\end{align*}
By homogeneity of $b_n(x,\xi)$ and of estimate \eqref{eqn:upperbound_est}, the same estimate holds for all $\|\xi\|\ge 1$. Thus we are done. 
\end{proof}

We can now apply this estimate to study the essential spectral radius of the transfer operator.

\begin{proof}[Proof of Theorem \ref{thm:essential_spectral_gap}.] To begin, we assume that $\mu^{-1}$ is coexpanding on average. 
Recall that by definition $\mc{G}(\phi)=\int \phi\circ f_{\omega}\,d\mu(\omega)$. As in equation \eqref{eqn:averaged_operator_action}, we also have the action on operators, which we denote by $\mc{L}$. From before, we are interested in $\|\mc{G}^n\phi\|_{-s}$. We will take $n$ to be some potentially large number to be chosen later.
Then using a version of Jensen's inequality for Hilbert spaces (\cite[Thm.~1.1]{perlman1974jensens})
to pass to the second estimate, we find that:
\begin{align*}
\|\mc{G}^n\phi\|_{-s}^2=\|\Delta^{-s}\int \phi\circ f\,d\mu^n(f)\|_0^2
\le \int \|\Delta^{-s} (\phi\circ f)\|^2_0\,d\mu^n(f).
\end{align*}
But due to volume preservation,
\begin{align}
\label{L2Isom}
 \int \|\Delta^{-s} (\phi\circ f)\|^2_0\,d\mu^n(f)&=
  \int \|(\Delta^{-s} (\phi\circ f))\circ (f^{-1})\|^2_0\,d\mu^n(f)\\
  &=\int \|((\Delta^{-s})^f \phi\|^2_0\,d\mu^n(f). \nonumber
\end{align}
By Lemma \ref{lem:non_negative_sum},  there exists a pseudodifferential operator $B\in \Psi^{-s}$ and a compact operator $\mc{K}_1$ such that
$\displaystyle
\|\mc{G}^n\phi\|_{-s}^2\le \|B\phi\|_0^2+\langle \mc{K}_1\phi,\phi\rangle ,
$
and 
\begin{equation}\label{eqn:bound_on_B_symbol}
\abs{\sigma_B(x,\xi)}^2\le (1+\epsilon) \int \abs{\sigma_{(\Delta^{-s})^f}(x,\xi)}^2\,d\mu^n(f).
\end{equation}

We now compare the symbols of $B$ and $\Delta^{-s}$. For any $0<\lambda<1$, as long as $n$ is sufficiently large, by Lemma \ref{lem:action_on_symbol} applied to the right hand side of equation \eqref{eqn:bound_on_B_symbol}, it follows that $\abs{\sigma_{B}}\le\lambda \abs{\sigma_{\Delta^{-s}}}$ restricted to frequencies $\abs{\xi}>r$ for some $r$. 

We now conclude using the symbol comparison lemmas. As $\abs{\sigma_{B}}\le\lambda \abs{\sigma_{\Delta^{-s}}}$, it follows from Lemma \ref{lem:comparison_of_symbols_and_operators} applied to $B$ that for all $\epsilon>0$ there exists a compact, smoothing operator $\mc{K}_{\epsilon}$ such that 
\[
\|\mc{G}^n\phi\|_{-s}^2\le \|B\phi\|_{0}^2+\langle \mc{K}_1\phi,\phi\rangle\le (\lambda+\epsilon)\|\Delta^{-s}\phi\|_0^2+\langle (\mc{K}_\epsilon+\mc{K}_1)\phi,\phi\rangle
\]
Recalling that $\|\Delta^{-s}\phi\|_0^2=\|\phi\|_{-s}^2$, we then find by Lemma \ref{lem:move_compact_terms} that there exists a compact operator $\mc{K}_2$ such that 
\[
\|(\mc{G}^n+\mc{K}_2)\phi\|^2_{-s}\le (\lambda+2\epsilon)\|\phi\|_{-s}^2,
\]
which establishes essential spectral gap since $\lambda+2\epsilon<1$ if $\lambda$ and  $\epsilon$ are sufficiently small.

 In the remaining case, if $\mu$ is coexpanding on average, then the dynamics given by $\phi\mapsto \int \phi\circ f^{-1}\,d\mu$ has essential spectral gap on $H^{-s}$ by the above. Hence as the adjoint action is given by $\phi\mapsto \int \phi\circ f\,d\mu$, the random dynamics of $\mu$ have an essential spectral gap on $H^s$ for small $s>0$. This is because an operator and its adjoint have the same essential spectral radius.
\end{proof}

We note that the proof given above in fact shows the following {\em Lasota-Yorke} inequality under the assumption that 
$\mu^{-1}$ is coexpanding on average: 
\begin{equation}
\label{LY} 
\|\mc{G}^n\phi\|_{-s}\leq \eta^n \|\phi\|_{-s}+C_n \|\phi\|_{-\bar s} 
\end{equation}
where $\bar s=s+\frac{1}{2}>s$ and  the constants $\eta$ and $C$ are uniform in some neighborhood of $\mu.$ This
estimate will be useful in the next section.

\section{Applications of essential spectral gap}
\label{ScApplications}

\subsection{Essential spectral gap on \texorpdfstring{$L^2$}{square integrable functions}}
  The proof of the following theorem uses the interpolation results recalled in Subsection \ref{subsec:interpolation}.

\begin{theorem}\label{thm:essential_spectral_gapL2}
Suppose that $\mu$ and $\mu^{-1}$ are both coexpanding on average measures on $\Diff^{\infty}_{\vol}(M)$ with compact support. Then there exists $s_0>0$ such that the induced action on $H^s(M)$ has essential spectral gap for all $s\in [-s_0,s_0]$.
\end{theorem}

\begin{proof}[Proof of Theorem \ref{thm:essential_spectral_gapL2}.]
To begin, note that if $\mc{G}_\mu$ denotes the generator of $\mu$ then  $\mc{G}_\mu^*=\mc{G}_{\mu^{-1}}$. Now by Theorem \ref{thm:essential_spectral_gap} for some small $s_0>0$, two things follow:  because $\mu$ is coexpanding on average $\mc{G}_{\mu}$ has essential spectral gap on $H^{s_0}$, and because $\mu^{-1}$ is coexpanding on average it follows that $\mc{G}_{\mu}$ has essential spectral gap on $H^{-s_0}$. Now we can apply Lemma \ref{lem:interpolation}, and interpolate between $H^{-s_0}$ and $H^{s_0}$ to get the Sobolev space $H^s$ for any $s\in (-s_0,s_0)$ by equation \eqref{eqn:complex_interpolation}. By the lemma, the interpolated operator $\mc{G}_{\mu}$ has essential spectral gap on $H^s$ as long as it has it on $H^{-s_0}$ and $H^{s_0}$. All that one needs to check is that the interpolated operator is indeed the operator given by the composition with the dynamics, but this is clear because $C^{\infty}$ functions are dense in $H^{s_0}$, $H^{-s_0}$, and $H^s$.
\end{proof}

\begin{remark}
\label{rem:besov}
Note that due to \eqref{eqn:besov} under the hypotheses of Theorem~\ref{thm:essential_spectral_gap} we also obtain spectral gap on the Besov spaces $B^{s}_{2 q}$ for $q\ge 1$ and $s\in [-s_0,s_0]$.
\end{remark}

\subsection{Pair correlation}
Recall that a measure preserving map $F$ is {\em totally ergodic} if $F^q$ is ergodic for all $q\in \mathbb{N}$.

\begin{theorem}\label{ThEMAnnealed} Let $\mu^{-1}$ be a coexpanding on average measure on $\Diff^\infty_{\vol}(M)$ with compact support.

(a)
 Suppose that the measure $\mu$ is weak mixing in the sense explained in \S\ref{SSWMSkew}. Then the random walk defined by $\mu$ on $M$ is exponentially mixing. Specifically, there exists $s>0$, $C>0$, and $0<\lambda<1$ such that for $\phi\in H_0^s$ and $\psi\in H^{-s}_0$, the Sobolev spaces of zero mean,
\[
\abs{\langle \phi,\mc{\mc{G}}^n \psi\rangle }\le C\lambda^n\|\phi\|_{H^s}\|\psi\|_{H^{-s}}.
\]
In particular, this implies that for any fixed $\alpha>0$, and zero mean $\phi,\psi\in C^{\alpha}(M)$,
\[
\abs{\langle \phi,\mc{G}^n \psi\rangle}\le C\lambda^n \|\phi\|_{C^{\alpha}}\|\psi\|_{C^{\alpha}}. 
\]

(b) The same conclusion holds if we only assume that the skew product defined by \eqref{DefSkew} is totally ergodic.
\end{theorem}

\begin{proof}
(a) By Theorem \ref{thm:essential_spectral_gap} there exists $s>0$ such $\mc{G}$ acting on $H^{-s}$ has essential spectral gap. From the spectral decomposition theorem, e.g. \cite[Sec.~148]{riesz1990functional}, we can decompose $H^{-s}$ into two $\mc{G}$ invariant pieces $H_1$ and $H_2$ so that $H_1$ contains the part of the spectrum of modulus at least $1$ and the action on $H_2$ has 
 has spectral radius smaller than some $\eta<1$. 
There is a corresponding invariant decomposition in the dual space $H^{s}$ for the action of the adjoint $\mc{G}^*$, which we denote $H_1^*$ and $H_2^*$. Note that $H_1$ and $H_1^*$ are finite dimensional from the assumption of essential spectral gap. 
Given $\phi\in H^{s}_0$ and $\psi\in H^{-s}_0$, decompose $\phi=\phi_1+\phi_2$, $\psi=\psi_1+\psi_2$
where $\phi_1\in H_1^*,$ $\phi_2\in H_2^*$, $\psi_1\in H_1$ and $\psi_2\in H_2$.
\begin{equation}
 \label{OrtSpec}   
\abs{\langle \phi,\mc{G}^n\psi\rangle}\le \abs{\langle \phi_1,\mc{G}^n \psi_1\rangle}+\abs{\langle \phi_2,\mc{G}^n \psi_2\rangle}
\end{equation}
Any element of $H_1^*$ is an element of $L^2$ that satisfies 
$\mc{G}^{n*}{\phi_1}= r e^{i\theta}{\phi_1}$ for some real  $r\geq 1$ and
$\theta$. As the $f$ preserve volume, this adjoint is given by $\phi\mapsto \int \phi\circ f^{-1}\,d\mu(f)$. 
Arguing as in \S \ref{SSWMSkew} we get that $r=1$ and ${\phi_1\circ f^{-1}}=e^{i\theta}{\phi_1}$
for $\mu$ almost every $f$.
Now our assumption about weak mixing implies that
$\phi_1$ must be constant, and hence $0$ by assumption of zero integral. 

For the second term in \eqref{OrtSpec} 
we have exponential decay because the norm of $\mc{G}^n$ on $H_2$ is at most $\eta^n $ for large $n$.
Thus 
$\displaystyle
\abs{\langle \phi,\mc{G}^n\psi\rangle}\le \eta^n\|\phi\|_{H^s}\|\psi\|_{H^{-s}_0}, 
$  as desired.

(b) Suppose now that $F\colon \Sigma\times M\to \Sigma\times M$ is ergodic but $\mathcal{G}$ does not have a spectral gap. It is easy to see that the set eigenvalues corresponding to eigenfunctions
depending only on the $M$ coordinate form an abelian group 
(cf.~\cite[Theorem 12.1.1(1)]{CFS}).
Since $\mathcal{G}$ has an essential spectral gap on $H^{-s}$ the space of eigenfunctions in $H^{-s}$ and hence in $L^2$ is finite dimensional.
Therefore, the aforementioned group is finite, and so there exists $q$ such that all eigenfunctions are $q$-th roots of unity. It follows that all 
eigenfunctions are invariant by $F^q$ and so $F^q$ is not ergodic.
\end{proof}

\subsection{Multiple mixing.}
We can now check multiple mixing. 

\begin{corollary}\label{cor:multiple_mixing}
Under the assumptions of Theorem \ref{ThEMAnnealed}, there exists 
a constant $0\!\!<\!\!\theta\!\!<\!\!1$ such that 
for all  $d\in\mathbb{N}$, there is a constant $C$ such that 
for all zero mean $\phi_0, \phi_1\dots \phi_{d-1}\in C^1(M)$, all zero mean $\phi_{d} \in H^{-s}$,
and for all $0\!=\!n_0\!<\!n_1\!<\!\dots \!<\!n_d$
we have:
\[ \left|\mathbb{E}_\mu \left[\int \phi_0 \mathcal{G}^{m_1} \left( \phi_1 \left(\mathcal{G}^{m_2} \phi_2\dots \phi_{d-1}\left(
\mathcal{G}^{m_d} \phi_d \right)\right)\right) dx \right]  \right|\leq C \theta^{L} 
\left[\prod_{j=0}^{d-1} \|\phi_j\|_{C^1}\right] \|\phi_d\|_{H^{-s}}, \]
where $m_j=n_j-n_{j-1}$ and $\displaystyle L=\min_j m_j.$
\end{corollary}

\begin{proof}
We proceed by induction. For $d=1$ the result holds due to  Theorem \ref{ThEMAnnealed}.   

For $d>1$, let $\psi=\phi_{d-1} \mathcal{G}^{m_d} \phi_d.$ Note that in the proof of Theorem \ref{ThEMAnnealed} we established that
$H_1$ is trivial, since the only eigenfunction of modulus $1$ is $1$, which is orthogonal to zero mean functions.
So $\|\mathcal{G}^{m_d} \phi_d\|_{H^{-s}}{\leq}
C_1\theta^{m_d}\|\phi_d\|_{H^{-s}}$. Since multiplication by a $C^1$ 
function is a bounded operator on $H^{-s}$, with the norm bounded by the $C^1$ norm of the function,  $\|\psi\|_{H^{-s}}{\leq} C_2\theta^{m_d}\|\phi_{d-1}\|_{C^1}\|\phi_d\|_{H^{-s}}$. 
$\psi$ need not have zero mean but we can split it as $\psi_1+\psi_2$ where 
$\psi_1=\langle \psi, 1 \rangle 1$ and
$\psi_2$ has zero mean. Hence applying the inductive assumption for $d-2$ and $d-1$ respectively, and noticing that
$$\langle \psi, 1\rangle=\int \phi_{d-1} \mathcal{G}^{m_d} {\phi_{d}} \; dx=O\left(\|\phi_{d-1}\|_{C^1} \|\phi_d\|_{H^{-s}} \theta^{m_d} \right) $$
we obtain the result.
\end{proof}

\subsection{Non-mixing systems.}

Without assuming ergodicity we have the following consequence of coexpansion on average.

\begin{corollary}\label{cor:coexpanding_components}
 Consider a measure $\mu$ on $\Diff^\infty_{\vol}(M)$ with compact support, and suppose that $\mu^{-1}$  or $\mu$ is a coexpanding on average. Then there exists $q>0$
 and a finite collection of disjoint positive measure subsets $M_1, \dots ,M_l$ of $M$ of total measure $1$ such that for each $j$,
 $F^q$ preserves $\Sigma\times M_j$ and the restriction of $F^q$ to this set is totally ergodic.
\end{corollary}

\begin{proof}
As $\mu$ and $\mu^{-1}$ have the same ergodic components, we will just consider the case where $\mu^{-1}$ is coexpanding on average.
    If $F$ is totally ergodic, there is nothing to prove, so we suppose that $F$ is not totally ergodic. Then there exists $q$ such that $F^q$ is not ergodic. 
     As is standard for random systems, we say that a set is invariant if it is invariant modulo $\vol$-null sets. 
    By Proposition~\ref{PrRandomET}, there is a set $\tilde M\subset M$ which is invariant by almost all $f_\omega^q$.
    Note that the space of functions depending only on $x$ which are invariant mod zero  by almost all $f_\omega^q$ is finite dimensional
    (its dimension does not exceed the dimension of $H^1_*$ from the proof of Theorem \ref{ThEMAnnealed}).
 Hence the $\sigma$-algebra of invariant sets is finitely generated, so there are finitely many sets 
    $\tilde M_1, \tilde M_2,\dots \tilde M_{\tilde l}$ which are invariant and such that the restriction of $F^q$ to 
    $\Sigma\!\times\!\tilde M_j$ is ergodic. If $F^q$ is totally ergodic restricted to these sets, we are done.
    Otherwise there is $\hat q>1$ such that, applying Proposition~\ref{PrRandomET} again,
    we could split 
    $\tilde M_j\!=\!\hat M_{j1}\bigcup \dots \bigcup\hat M_{jk_j}$ so that
    $\hat M_{ji}$ are invariant under $F^{\hat q}$ and $F^{\hat q}$ is ergodic on 
    $\Sigma\!\times\!\hat M_{ji}$,
     and the splitting is non trivial in the sense that at least one $\tilde M_j$
    is split into more than one  piece.
    
    Continuing this procedure we obtain finer and finer subpartitions of $M.$ Since the number of elements in every 
    partition is at most dimension of $H_1^*$, this process stops after finitely many steps.
\end{proof}

\subsection{Stability of mixing.}
\label{SSStability}
Let $\cK$ be a compact set of measures $\mu$ such that $\mu^{-1}$ is coexpanding on average measures and   \eqref{LY} holds for $\mu\in \cK$.
Consider the following Wasserstein type function on the space of measures
$$ \dd(\mu, \tilde \mu)=\inf \int_\pi  \sqrt{[d_{C^2}(f, \tilde f)+d_{C^2}(f^{-1}, \tilde f^{-1})]} d\pi(f, \tilde f).
$$
where the infimum is over all measures $\pi$ with marginals $\mu$ and $\wt{\mu}.$

\begin{theorem}\label{thm:exponential_mixing}
\label{ThSpStab}
Suppose $\mu\in \cK$ is a measure on $\Diff^\infty_{\vol}(M)$ with compact support and such that
the associated operator $\mc{G}$ has no eigenvalues on the unit circle  in $H^{-s}_0$, the space of zero mean distributions. Then the same holds for any measure 
$\tilde \mu$ which is sufficiently close to $\mu$ with respect to $\dd$ metric.
\end{theorem}

We start with some notation. Let $\tilde \mu$ be a measure such that $\delta:=\dd(\mu, \tilde \mu)$ is small. 
Denote $G:=\mc{G}_\mu,$ $\tilde G:=\widetilde{\mathcal{G}}_{\tilde \mu}.$
We need an auxiliary estimate. 

\begin{lemma} 
\label{LmWH}
(a) There is a constant $K$ such that for $s, \brs$ from \eqref{LY}
$$ \|G\phi-\tilde G\phi\|_{-\bar s}\leq K \delta \|\phi\|_s. $$

\noindent
(b) For each $n$ there is a constant $K_n$ such that 
$\displaystyle \|G^n\phi-\tilde G^n\phi\|_{-\bar s}\leq K_n \delta \|\phi\|_s. $
\end{lemma}

\begin{proof}
(a) For Sobolev spaces of positive indices this is proven in \cite[Lemma~2.39]{BaladiBook}. The result for negative 
indices follows by duality. Namely,
given $\psi\in H^{\bar s}$ we have

$$ \left|\langle G\phi-\tilde G \phi, \psi \rangle\right|=
\left|\langle \phi, G^*\psi -\tilde G^*\psi\rangle\right|\leq \|\phi\|_s \|(G^*-\tilde G^*)\psi\|_s
$$
The second factor can be rewritten as
$$ \|(G^*-\tG^*)\psi\|_{-s}=\left\Vert \int \left[\psi\circ f^{-1}-\psi\circ \tilde f^{-1} \right] d\pi \right\Vert_{-s}
\leq  \int \left\Vert\psi\circ f^{-1}-\psi\circ \tilde f^{-1} \right\Vert_{-s} d\pi
$$$$
\leq \int K \sqrt{d_{C^2} (f^{-1}, \tilde f^{-1})} d\pi \|\psi\|_{-\bar s} \leq K \delta \|\psi\|_{-\bar s} $$
 proving part (a).

\noindent
(b) follows from (a) by writing $\displaystyle G^n-\tG^n=\sum_{j=0}^{n-1} \left[G^{n-j} \tG^j-G^{n-j-1} \tG^{j+1}\right]. $
\end{proof}

\begin{proof}[Proof of Theorem \ref{ThSpStab}]
This follows from Proposition \ref{PrSpecCont}. Indeed \eqref{UniGrowth} follows from the 
inequality
$\displaystyle \|\cG^n\|_{-s}\leq C\max_f \|Df\|^{-s} 
$
which can be obtained by interpolation. \eqref{KL-LY} follows from \eqref{LY},
\eqref{FinMult} holds due to Theorem \ref{thm:essential_spectral_gap}, and \eqref{TameCont}
holds by Lemma \ref{LmWH}.
\end{proof}

\subsection{Genericity of exponential mixing.}
\label{ScEMGeneric}
Using the decomposition from Corollary~\ref{cor:coexpanding_components}, we can show that if the coexpanding on average condition is generic among tuples, then so is ergodicity. Recall that we associate with a tuple the random dynamical system that assigns equal weight to each element of the tuple.

\begin{proposition}\label{prop:stable_exponential_mixing}
Suppose that the coexpanding on average condition is dense in $\Diff^{\infty}_{\vol}(M)^m$, the space of $m$-tuples, then stable exponential mixing is dense in the space of $(m+1)$-tuples. 
\end{proposition}

\begin{proof}
 By assumption, the coexpanding on average condition is dense among $m$-tuples. From this it follows that 
 the property that $(f_1,\ldots,f_m)$ is coexpanding on average both forwards and backwards is dense. Thus 
 by Theorem~\ref{thm:essential_spectral_gapL2} the operator $\mc{G}$ associated to the full tuple has essential spectral gap on $L^2$, as it is the average of two operators of norm $1$, one having this property.

From Corollary \ref{cor:coexpanding_components} applied to $(f_1,\ldots,f_m)$ we see that there exists $q$ and a finite partition $\{M_i^f,\ldots,M_{l}^f\}$ such that the restriction of the $q$-th power of the dynamics of $(f_1,\ldots,f_m)$ to this set is totally ergodic. 
 Similarly to the proof of Theorem \ref{ThEMAnnealed}, once we have an essential spectral gap,
in order to obtain exponential mixing, it suffices to show that in fact every power of the dynamics generated by $(f_1,\ldots,f_m,g)$ is ergodic. Hence we must show that for each power of the dynamics, the $\sigma$-algebra of a.s.~invariant sets is trivial. This  $\sigma$-algebra is a coarsening of the algebra $\cI$ generated by the partition 
$\{M_i^f\}_{1\le i\le l}$.
Since $\cI$ is finite, we need to show that for each nontrivial $A\in \cI$, a generic map does not preserve $A$,
but this follows from Proposition \ref{PrNoInvSetsDiff}.
\end{proof}

The argument presented above can be applied to show the genericity of exponential mixing in other settings as well.

\begin{proof}[Proof of Theorem \ref{ThGenericNearIsom}]
Suppose that $(R_1,\ldots,R_{m-1})$ generates $ \mathrm{Isom}(M)$. By Theorem \ref{thm:coa_isometry_perturbation},
its perturbation is either isometric or (generically) coexpanding on average. 
Suppose we extended the tuple with an extra map $(R_1,\ldots,R_{m-1} ,R_{m})$, and then perturbed to obtain a tuple $(f_1,\ldots,f_{m-1},f_{m})$. If $(f_1,\ldots,f_{m-1})$ is not simultaneously conjugated back to isometries, then this tuple is coexpanding on average forwards and backwards. Hence  
 Proposition \ref{PrNoInvSetsDiff} shows that possibly after 
 a further $C^{\infty}$ small perturbation $\wt{f}_{m}$ of $f_{m}$ the resulting dynamics of $(f_1,\ldots,f_m,\wt{f}_{m+1})$ is stably exponentially mixing. As generating tuples $(R_1,\ldots,R_{m-1})$ are dense in $\mathrm{Isom}(M)$ by \cite[Thm.~1.1]{field1999generating}, Theorem \ref{ThGenericNearIsom} follows.
\end{proof}

\subsection{Dissipative perturbations.}
\label{SSDissipative}
 Due to the spectral gap, small dissipative perturbations of a measure $\mu$ with $\mu^{-1}$ coexpanding on
 average and conservative must have an absolutely continuous invariant measure with a density in $H^s,s>0$. 
 We note that in 
 \cite[Thm.~A]{brown2024absolute} the authors exhibit an open set of (co)expanding on average random systems
 on $\mathbb{T}^2$ such that 
\begin{enumerate}[leftmargin=*]
\item[(i)]
There is an absolutely continuous stationary measure;
\item[(ii)]
Any stationary measure is either absolutely continuous or finitely supported.
\end{enumerate}

 They conjecture that the same conclusion holds for arbitrary mildly dissipative expanding on average
 systems on surfaces.

Our result below extends (i) to arbitrary dimension (for coexpanding systems). However, our methods do not give (ii) even in dimension two since the {\it a priori} regularity of non-atomic stationary measures 
obtained in 
\cite{BrownRodriguezHertz} is insufficient to conclude that the measure is in $H^{-s}$ for small $s.$ 
 
We also note that for higher dimensional systems there could be a  stationary measure supported on proper submanifolds. A simple example is provided by a $k$ point motion 
discussed in \S \ref{subsec:products} which preserves generalized
diagonals. It is an important open question if fractal stationary measures are also possible in either
the conservative or mildly dissipative setting (cf. \cite[Conjecture 1.1.12]{brown2025measure}).
 
\begin{theorem}
Let $\mu$ be a measure on $\Diff^{\infty}_{\vol}(M)$ such that  $\mu^{-1}$ is coexpanding on average.
Then there exists $\delta>0$ such that if $\tilde\mu$ is a  $C^1$ small perturbation of $\mu$ 
supported on diffeomorphisms in $\Diff^{\infty}(M)$ satisfying that for each $x\in M$
\begin{equation}
 \left|\det(D_x f)-1\right|\leq \delta,
\label{NearCon}
\end{equation}
then: 
\begin{enumerate}[leftmargin=*]
\item[(a)] The generator $\widetilde{\cG}$ of $\tilde\mu$ process
 has an essential spectral spectral radius smaller than 1 
 in $H^{-s}$ 
for some small $s>0$.
\item[(b)] The random system generated by $\wt{\mu}$ has an absolutely continuous invariant measure in $H^s$ 
for some small $s>0$.
\end{enumerate}
\end{theorem}

\begin{proof}
To prove (a) we note that the only place where the volume preservation was used in the proof of Theorem \ref{thm:essential_spectral_gap}
is \eqref{L2Isom} where we used that the composition with $f^n_\omega$ preserves $L^2$-norm. Under the volume preservation
assumption \eqref{NearCon}, the norm of the composition on $L^2$ is increased by at most a factor of $(1+\delta)^n$ which is sufficient for the argument as long as
$(1+\delta)\eta<1$.

To prove (b) note that  $\widetilde{\cG} 1=1$ and by part (a), 1 is an eigenvalue of finite multiplicity.
It follows that it is also eigenvalue of finite multiplicity of the adjoint operator $\widetilde{\cL}$,
 which acts on $H^s$, $s>0$. 
In particular, 
there exists an $\widetilde\cL$ invariant function  $\phi\in H^s.$ 
Multiplying $\phi$ by $i$  if necessary we may assume that $\Re(\phi)\neq 0$.
Since $\Re(\phi)$ is preserved by $\widetilde\cL$ we may assume from the beginning that $\phi$ is real.
By the same argument we may assume that $\phi^+:=\max(\phi, 0)$ is not identically zero. 
We claim that $\phi^+\leq \widetilde\cL\phi^+.$ Indeed if
$\phi(x)>0$ then 
$\DS (\widetilde\cL\phi^+)(x)\geq (\widetilde\cL\phi)(x)=\phi(x)=\phi^+(x).$
On the other hand if $\phi(x)\leq 0$ then $(\widetilde\cL\phi^+)(x)\geq 0=\phi^+(x)$
proving the claim.
The claim implies that
$$ \langle \phi^+, 1 \rangle\leq  \langle \widetilde\cL\phi^+, 1 \rangle=\langle\phi^+, \widetilde\cG 1\rangle
=\langle\phi^+, 1\rangle. 
$$
But this is only possible if the inequality is in fact equality, that is, $\widetilde\cL \phi^+=\phi^+.$
Thus the measure with density $\phi^+$ is a stationary measure of our Markov chain.
\end{proof}

\begin{remark}
 Of course, if our random system is totally ergodic, then by Keller--Liverani stability result, all 
$\widetilde\cL$ invariant functions 
 (real or complex) are proportional, so in that case the stationary measure
    is unique.
\end{remark}

\subsection{Central limit theorem.} In this subsection, we deduce the central limit theorem from the spectral gap.

\begin{theorem}\label{thm:central_limit_theorem}
Suppose that $M$ is a closed manifold and that $\mu$ is a compactly supported measure on $\Diff^{\infty}_{\vol}(M)$ such that $\mu^{-1}$ is coexpanding on average and weak mixing. Then 
\begin{enumerate}[leftmargin=*]
\item[(a)] (CLT) 
The associated random dynamical system satisfies the 
central limit theorem. Namely let $\phi\colon M\to \R$ be a zero mean H\"older function.  Then for $z\in \R$, 

$$ \lim_{N\to\infty} (\mu^N\otimes \vol)\left((\omega,x): \sum_{n=0}^{N-1} \phi(f_\omega^n x)\leq \sqrt{N} z\right)=\int_0^z g_\sigma(s) ds
$$
where 
\begin{equation}
\label{GreenKubo}
 \sigma^2=\sigma^2(\phi)=\|\phi\|_{L^2}^2 +2\sum_{n=1}^\infty \langle \phi, \cG^n\phi \rangle_{L^2},
 \end{equation}
\hskip-8mm and $g_\sigma$ is the density of the normal random variable with zero mean and variance~$\sigma^2.$
\vskip1mm

\item[(b)] (Berry--Esseen bound) Moreover there is a constant $K$ such that if $\phi$ is a zero mean $C^1$ function 
with $\sigma^2(\phi)\neq 0$, then for all $z\in \R$,
\[ \left|(\mu^N\otimes \vol)\left((\omega,x): \sum_{n=0}^{N-1} \phi(f_\omega^n x)\leq \sqrt{N} z\right)- \int_0^z g_\sigma(s) ds
\right|\leq \frac{K}{\sqrt N}.  \]

\item[(c)] If, in addition, $\mu^{-1}$ is also coexpanding on average, then both CLT and Berry--Essen bound hold for $L^\infty$ observables.

\end{enumerate}
\end{theorem}

\begin{proof} 

Part (b) follows by \cite[Theorem 3.7]{Gouezel} which says that the Berry--Esseen bound holds provided that
$\mc{G}$ has spectral gap on some Banach space $\mathbb{B}$, $1$ is a simple eigenvalue of $\cG$, and, denoting 
$\DS \cG_t (\psi)=\cG\left(e^{it \phi} \psi\right)$, we have that the map $t\mapsto \mc{G}_t$ is $C^3$ in the strong operator norm in $\Aut(\mathbb{B}).$
Take $\mathbb{B}=H^{-s}$. Then the essential spectral gap holds 
by Theorem \ref{thm:essential_spectral_gap}, the second condition holds  due to Theorem \ref{ThEMAnnealed}
since $\mu$ is weak mixing, and the last condition holds
because $e^{it \phi}$ is the sum of its Taylor series and multiplication by $\phi$, and hence $\phi^k$, define bounded operators in $H^{-s}$ with at most exponentially growing norms.

Next, under the assumption of part (c), the generator has a spectral gap on $L^2$ by Theorem \ref{thm:essential_spectral_gapL2}.
Now part (c) follows from \cite[Theorem 3.7]{Gouezel}, this time with 
$\mathbb{B}=L^2$, and the fact that multiplication by an $L^\infty$ function is a bounded
operator on $L^2.$

Part (a) follows from part (b) and Proposition \ref{PrDenseCLT} below.
\end{proof}
\begin{proposition}
\label{PrDenseCLT}
Suppose that $x_n$ is a Markov process with state space $M$, and $\cB$ is a space of zero mean functions on $M$ where the generator $\mc{G}$ has summable
correlations in the sense that
$
\DS
\left|\langle \cG^n \phi, \psi\rangle\right|\leq a(n)\|\phi\|{ \|\psi\|}$
 with  $\sum_n a(n)<\infty$. 
If there is a dense set $\cD\subset \cB$ such that for all $\phi\in\cD$, 
$N^{-1/2}\sum_{n=0}^{N-1} \phi(x_n)$ converges in law as $N\to\infty$
to a normal random variable with zero mean and variance $\sigma^2(\phi)$ given by \eqref{GreenKubo}, 
then the same holds for all $\phi\in\cB$.
\end{proposition}

\begin{proof}
In the course of the proof we will denote $\displaystyle S_N(\phi)=\sum_{n=0}^{N-1} \phi(x_n)$ and
let $\Phi_\sigma(z)$ be the cumulative distribution function of the normal random
variable with zero mean and variance $\sigma^2$. In the Big-O terms below, the implied constants depend only on the $a(n)$ unless otherwise noted. 

Take $\phi\in \cB.$ Recall that $\DS \Vert S_n(\phi)\Vert_{L^2}^2=N\sigma^2(\phi)+O(\|\phi\|_{L^2}^2).$
If $\sigma^2(\phi)=0$, then $S_N(\phi)/\sqrt{N}$ converges to 0 due to the Chebyshev inequality.

Next, suppose that $\sigma^2(\phi)\neq 0.$ Take $z\in \mathbb{R}$ and $\eps>0$ and choose $\psi\in\cD$
such that $\|\eta\|\leq \eps^4$ where $\eta=\phi-\psi$. 

Then 
\begin{align*}
 \Prob\left(\frac{S_N(\psi)}{\sqrt{N}}\leq z-\eps\right)-\Prob\left(\left|\frac{S_N(\eta)}{\sqrt{N}} \right|\geq \eps\right)
 &\leq 
\Prob\left(\frac{S_N(\phi)}{\sqrt{N}}\leq z\right)  
\\
&\leq 
\Prob\left(\frac{S_N(\psi)}{\sqrt{N}}\leq z+\eps\right)+\Prob\left(\left|\frac{S_N(\eta)}{\sqrt{N}} \right|\geq \eps\right).
\end{align*}
Since a straightforward computation using the summability of the correlations gives $\sigma(\psi)=\sigma(\phi)+O_{\phi}(\eps^4)$, we see that for large $N$,
$$ \Prob\left(\frac{S_N(\psi)}{\sqrt{N}}\leq z+\eps\right)\leq \Phi_{\sigma(\psi)}(z+\eps)+\eps\leq \Phi_{\sigma(\phi)}(z)+C\eps.$$
Similarly,
\[
\DS \Prob\left(\frac{S_N(\psi)}{\sqrt{N}}\leq z-\eps\right)\geq  \Phi_{\sigma(\phi)}(z)-C\eps.
\]
Also since $\sigma^2(\eta)=O(\eps^4)$, the Chebyshev inequality tells us that 
$\DS \Prob(|\eta|>\eps)=O(\eps^2)$.\vskip1mm

Combining the above estimates gives 
$\DS \Prob\left(\frac{S_N(\phi)}{\sqrt{N}}\leq z\right)=\Phi_{\sigma(\phi)}(z)+O_{\phi}(\eps). 
$
Since $\eps$ is arbitrary the result follows.
\end{proof}

\subsection{Quenched properties.}
The results described so far pertain to the averaged (annealed) dynamics. However, if ergodic properties of the two point motion are well understood, one can derive quenched results,
which we discuss briefly in this subsection. We say that the random dynamics has {\em quenched exponential mixing} on a Banach space $\mathbb{B}$ of functions on $M$ 
if there exists a constant $\theta<1$ and  random variable $C(\omega)$ such that for almost all $\omega$ and all zero mean functions $\phi, \psi\in \mathbb{B}$ we have
\[
\left| \int \phi(x) \psi(f_\omega^n x) dx\right|\leq C(\omega) \theta^n \|\phi\|_{\mathbb{B}} \|\psi\|_{\mathbb{B}}.
\]
We say that the random dynamics satisfies the {\em quenched Central Limit Theorem} on $\mathbb{B}$ if there exists a quadratic form
$\cD$ on $\mathbb{B}$ which is not identically zero such that for almost every $\omega$ and all zero mean functions $\phi\in \mathbb{B}$, then 
if $x$ is uniformly distributed on $M$, the distribution of $S_N^{\omega}\phi (x)/\sqrt{N}$ converges in law to a normal random variable with zero mean and variance $\cD(\phi)$. 

\begin{theorem}
Suppose that $\mu$ is a measure on $\Diff^{\infty}_{\vol}(M)$ such that $\mu^{-1}$ is coexpanding on average and the two point system is totally ergodic. Then 
the random dynamics defined by $\mu$ enjoys quenched exponential mixing on $H^s$ for $s>0$ and the quenched Central Limit Theorem on $C^1$.
\end{theorem}

\begin{proof}
By Corollary \ref{CrEoAk} the two point motion is also coexpanding on average backwards. Hence by 
Theorem~\ref{thm:essential_spectral_gap} the two point motion has essential spectral gap on $H^{-t}$ for small positive $t$. 
From the total ergodicity assumption together with Theorem~\ref{ThEMAnnealed}(b) the generator has a spectral gap on $H^{-t}$ for small positive $t$. 
    Now the result follows from \cite{DWDQAn} which says that a spectral gap on $H^t$ implies 
    the quenched exponential mixing on $H^s$ for $s\!>\!0$ and quenched Central Limit Theorem on $C^r$ for $r\!>\!|t|.$
\end{proof}

\section{Back to the introduction}
\label{ScBack}
Here we explain how the results stated in the introduction follow from the main results of our paper. Theorem \ref{thm:essential_spectral_gap} was proven in Section \ref{ScMainProof}, while Theorem \ref{ThGenericNearIsom} was proven
in \S\ref{ScEMGeneric}. We now show the remaining results.

\begin{proof}[Proof of Corollary~\ref{CrSpGap-MEM}.]
Suppose that $\mu^{-1}$ is a coexpanding on average measure that is totally ergodic. Then any perturbation $\wt{\mu}$ of $\mu$ has the same properties by Theorem~\ref{ThSpStab}. Thus $\wt{\mu}$ is multiple exponential mixing by Corollary \ref{cor:multiple_mixing}, and satisfies the central limit theorem by 
Theorem  \ref{thm:central_limit_theorem}. 
\end{proof} 

\begin{proof}[Proof of Theorem \ref{ThExpMixGen}.] 
Let $\mathfrak{G}$ be the set of measures that are totally ergodic and such that $\mu^{-1}$ is coexpanding on average.
These measures are strongly chaotic as was explained above. 
Also $\mathfrak{G}$ is open by Theorem \ref{ThSpStab}.
To see that it is dense note that
it is proven in \cite{elliott2023uniformly} that
for each open set $\cU$ in the space of $\Diff_{\vol}^\infty(M)$ there exists a clean measure $\mu_0$
supported on $\cU$ 
(see also Theorem \ref{ThXF} of the present paper).
Thus for each measure $\mu$ on $\cU$ and each $\eps>0$ the measure 
$\mu_\eps=\eps \mu_0+(1-\eps) \mu$ belongs to $\mathfrak{G}$ by Corollary~\ref{rem:clean}.
Thus $\mathfrak{G}$ is dense.
\end{proof}

\begin{proof}[Proof of Corollary~\ref{cor:main_cor}.]
The fact that the examples of measures $\mu$ described in the corollary have $\mu^{-1}$  coexpanding on average and are totally ergodic 
follows from Propositions \ref{PrTorusAffine}, \ref{PrHomogeneous} and 
Theorem \ref{thm:coa_isometry_perturbation} respectively. The strong chaoticity follows from Corollary \ref{cor:multiple_mixing} and Theorem  \ref{thm:central_limit_theorem}. 
The same properties hold for $\mu^{-1}$ since $\mu^{-1}$ belongs to the same class as $\mu.$
Now the spectral gap on $L^2$ follows
from Theorem \ref{thm:essential_spectral_gapL2}.
\end{proof} 

\printbibliography
\end{document}